\documentclass[letterpaper,12pt]{amsart}
\usepackage[margin=1.19in]{geometry}
\usepackage{amsmath,amsthm,amssymb,mathtools,amsfonts,mathrsfs}
\usepackage{xspace,xcolor}
\usepackage[breaklinks,colorlinks=true,citecolor=teal,linkcolor=teal,urlcolor=teal,pagebackref,hyperindex]{hyperref}
\usepackage[alphabetic]{amsrefs}
\usepackage{tikz-cd}
\usepackage[all]{xy}
\usepackage{bbm}
\usepackage{array}
\usepackage{setspace}
\usepackage{makecell}
\usepackage{diagbox}
\usepackage{longtable}
\usepackage{fancyhdr}

\newcommand\Q{\mathbb Q}
\newcommand\Z{\mathbb Z}
\newcommand\V{\mathcal V}
\newcommand\W{\mathcal W}
\newcommand\F{\mathcal F}
\newcommand\RT{\mathcal RT}
\newcommand{\GM}{\mathbb C^*}
\newcommand\C{\mathbb C}
\newcommand\tw{\operatorname{tw}}

\newcommand\ev{\operatorname{ev}}

\newcommand\bg{\mathbf g}
\newcommand\bd{\mathbf d}

\newcommand{\Etw}{{E,\mathrm{tw}}}
\newcommand\Aut{\operatorname{Aut}}
\newcommand\PP{\mathbb P}
\newcommand\bM{\overline{\mathcal{M}}}
\newcommand{\vir}{{\operatorname{vir}}}
\newcommand{\Cont}{{\operatorname{Cont}}}
\newcommand{\GW}{{\operatorname{GW}}}
\newcommand{\Supp}{{\operatorname{Supp}}}
\newcommand{\tabincell}[2]{\begin{tabular}{@{}#1@{}}#2\end{tabular}} 

\newcommand{\gr}{{\mathrm{gr}}}
\newcommand{\GMR}{\mathsf{GMR}}
\newcommand{\BC}{{\mathbb{C}}}
\newcommand{\BD}{{\mathbb{D}}}
\newcommand{\BF}{{\mathbb{F}}}
\newcommand{\BQ}{{\mathbb{Q}}}

\newcommand{\BP}{{\mathbb{P}}}
\newcommand{\BZ}{{\mathbb{Z}}}
\newcommand{\ch}{\textup{ch}}
\newcommand\Ext{\operatorname{Ext}}
\newcommand\Hom{\operatorname{Hom}}
\newcommand{\CV}{\mathcal V}
\newcommand{\FM}{\mathfrak M}

\newtheorem{thm}{Theorem}[section]
\newtheorem{prop}[thm]{Proposition}
\newtheorem{conj}[thm]{Conjecture}

\newtheorem{lem}[thm]{Lemma}
\newtheorem{cor}[thm]{Corollary}
\newtheorem{rmk}[thm]{Remark}

\makeatletter
\let\@wraptoccontribs\wraptoccontribs
\makeatother

\begin{document}

\title[Poincar\'e polynomials of moduli of one-dimensional sheaves on $\PP^2$]{Poincar\'e polynomials of moduli spaces of one-dimensional sheaves on the projective plane}
% \author{\mainauthors, \smallskip \\  With an appendix by \contributors}
% \author{}

% \pagestyle{fancy}
% \fancyhf{}
% \fancyhead[LE,RO]{\mainauthors}
% \fancyfoot[C]{\thepage}
\author{Shuai Guo}
\address{School of Mathematical Sciences, Peking University}
\email{guoshuai@math.pku.edu.cn}

\author{Longting Wu}
\address{Department of Mathematics \& SUSTech International Center for Mathematics, Southern University of Science and Technology}
\email{wult@sustech.edu.cn}

\contrib[with an appendix by]{Miguel Moreira}
\address{Department of Mathematics, Massachusetts Institute of Technology}
\email{miguel@mit.edu}

\date{}

\maketitle

\begin{abstract}
Let $M_{\beta}$ denote the moduli space of stable one-dimensional sheaves on a del Pezzo surface $S$, supported on curves of class $\beta$ with Euler characteristic one. We show that the divisibility property of the Poincaré polynomial of $M_{\beta}$, proposed by
Choi-van Garrel-Katz-Takahashi follows from Bousseau's conjectural refined sheaves/Gromov-Witten correspondence. Since this correspondence is known for $S=\mathbb{P}^2$, our result proves Choi-van Garrel-Katz-Takahashi's conjecture in this case.

For $S=\mathbb{P}^2$, our proof also introduces a novel approach to computing the Poincar\'e polynomials using Gromov-Witten invariants of local $\mathbb{P}^2$ and a local elliptic curve.
Specifically, we compute the Poincar\'e polynomials of $M_{d}$ with degrees $d\leq 16$ and derive a closed formula for the leading Betti numbers $b_i(M_d)$ with $d\geq 6$ and $i\leq 4d-22$. We also propose a conjectural formula for the leading Betti numbers $b_i(M_d)$ with $d\geq 4$ and $i\leq 6d-20$. In the Appendix (by M. Moreira), a more general conjecture concerning the higher range Betti numbers of $M_{d}$ is presented, along with another conjecture that involves refinements from the perverse/Chern filtration.
\end{abstract}

\tableofcontents

\section{Introduction}
\subsection{Background} We work over the complex numbers $\C$.

The moduli space $M_{\beta,\chi}$ of
of one-dimensional Gieseker semistable sheaves $\F$ on a del Pezzo surface $S$, with $[\Supp(\F)]=\beta$ and Euler characteristic $\chi$, has attracted significant attention, as seen in works such as \cite{LeP93,Y14,Y15,Y18,Y23,Y24,CC15,CC16,CvGKT,Bou22,Bou23, MS23,PS23,KPS23,LMP24,klmp,PSSZ24}.
In this paper, we study
the Poincar\'e polynomial of the intersection cohomology of $M_{\beta,\chi}$.

By the work of Maulik-Shen \cite{MS23} and Yuan \cite{Y23}, we know that the Poincar\'e polynomial of the intersection cohomology of $M_{\beta,\chi}$ are independent of $\chi$. So we could simply set $\chi=1$ and denote the corresponding moduli space as $M_{\beta}$. $M_{\beta}$
is smooth with 
\[\dim_{\C} M_{\beta}=\beta^2+1.\]
We use 
$$P_{\beta}\left(y\right)\coloneqq\sum_{j=0}^{\operatorname{dim}_{\C} M_{\beta}} b_{2j}\left(M_{\beta}\right) y^{j}$$
to denote the Poincar\'e polynomial of $M_{\beta}$. Note that the odd Betti numbers of $M_{\beta}$ vanish by \cite{B95,M09}.
The Poincar\'e polynomials $P_{\beta}\left(y\right)$ were conjectured to have the following divisibility property:
\begin{conj}[Choi-van Garrel-Katz-Takahashi \cite{CvGKT}]\label{conj:main0}
Let $w=-K_S\cdot\beta$. $P_{\beta}\left(y\right)$ can be divided by the Poincar\'e polynomial of the projective space $\PP^{\omega-1}$, i.e.,
\[\frac{P_{\beta}\left(y\right)}{\sum_{j=0}^{\omega-1}  y^j}\]
is a palindromic polynomial of $y$.
\end{conj}
Here and below, we use $K_S$ to denote both the canonical divisor and the associated line bundle, depending on the context.

Conjecture \ref{conj:main0} is quite mysterious from the geometry point of view since $M_{\beta}$ is not expected to be a $\PP^{\omega-1}$-bundle in general. The above conjecture was proven in \cite{CvGKT} when arithmetic genus $p_a(\beta)$ is at most 2 together with a technical assumption\footnote{The technical assumption is that $\beta$ is either a line class, a conic class or a nef and big class. Actually Conjecture \ref{conj:main0} was originally stated in \cite{CvGKT} with this technical assumption.
After a private communication with Michel van Garrel and Jinwon Choi, we realize that this condition might not be necessary. So we remove this technical assumption from Conjecture \ref{conj:main0}.} on the curve class $\beta$. 
If we set $y=1$, a direct corollary from Conjecture \ref{conj:main0} is that the Euler characteristic $\chi(M_{\beta})$ should be divided by $w$. The divisibility of  $\chi(M_{\beta})$ was proven by Bousseau in the case of $S=\PP^2$ using a new sheaves/Gromov-Witten correspondence \cite{Bou23}.
A heuristic explanation from the perspective of physics was also provided in \cite[Section 1.5.1]{Bou23} to support Conjecture \ref{conj:main0}.

\subsection{Results overview} In this paper, we first want to 
show that Conjecture \ref{conj:main0} can be deduced from the refined sheaves/Gromov-Witten correspondence proposed by Bousseau \cite[Conjecture 8.16]{Bou20}:

% the following two key correspondences: the refined sheaves/Gromov-Witten correspondence proposed by Bousseau \cite[Conjecture 8.16]{Bou20}, and the all-genus local/relative correspondence established in \cite{BFGW}. Let us give a brief review of the refined sheaves/Gromov-Witten correspondence  at first.

Let
$$\Omega_{\beta}\left(y^{\frac{1}{2}}\right)\coloneqq y^{-\frac{1}{2}\operatorname{dim}_{\C} M_{\beta}}\sum_{j=0}^{\operatorname{dim}_{\C} M_{\beta}} b_{2j}\left(M_{\beta}\right) y^{j}$$
be the shifted Poincar\'e polynomial of $M_{\beta}$. Let 
$$\bar{F}^{N S}\coloneqq\sqrt{-1} \sum_{\beta>0} \sum_{k \in \mathbb{Z}_{>0}} \frac{1}{k^2} \frac{(-1)^{\beta^2+1}\Omega_\beta\left(y^{\frac{k}{2}}\right)}{y^{\frac{k}{2}}-y^{-\frac{k}{2}}} Q^{k \beta}$$
and set
\[\bar{F}^{S / E}\coloneqq\sum_{\beta>0}{(-1)^{\beta\cdot E-1}}  \sum_{g \geq 0}  \frac{N_{g,\beta}^{S/E}}{\beta\cdot E} \hbar^{2g-1} Q^{\beta}\]
to be the generating series of the maximal contact Gromov-Witten invariants of the pair $(S,E)$. See Section \ref{Section:loc/rel} for a precise definition of $\bar{F}^{S / E}$. The refined sheaves/Gromov-Witten correspondence proposed by Bousseau states that
\begin{conj}[Bousseau \cite{Bou20}]\label{conj:rsh/GW}
Under the change of variables $y=e^{i\hbar}$,
\[\bar{F}^{N S}=\bar{F}^{S/E}.\]
\end{conj}
The conjecture was proven in \cite{Bou23} when $S=\PP^2$ under the assumption of the $\chi$-independent conjecture of moduli of one-dimensional sheaves on $\PP^2$. The $\chi$-independent conjecture was later proven by Maulik and Shen \cite{MS23} for toric del Pezzo surfaces, and further generalized by Yuan \cite{Y23} to all del-Pezzo surfaces.

Our first result is that 

\begin{thm}\label{thm:main0}
\[\text{Conjecture \ref{conj:rsh/GW} }\Rightarrow \text{Conjecture \ref{conj:main0}.}\]
In particular, Conjecture \ref{conj:main0} holds when $S=\PP^2$.
\end{thm}

The proof relies on the all-genus local/relative correspondence, established by the authors in collaboration with Bousseau and Fan in \cite{BFGW}. In this paper, we introduce a new way to package 
the all-genus local/relative correspondence (by summing over genera):

\begin{prop}[=Proposition \ref{prop:local/rel}]
Fixing a curve class $\beta>0$, we have
\begin{equation*}
\W^{K_S}_\beta(\hbar)  =   \    \frac{(\beta\cdot E) \V^{S/E}_{\beta}(\hbar)}{ 2 \sin \frac{(\beta\cdot E) \hbar}{2} } 
    + \sum_{0< d_E[E]\leq \beta}  \sum_{\boldsymbol{\beta}=(\beta_1,\cdots,\beta_n)\atop |\boldsymbol{\beta}|+d_E[E] =\beta} \frac{1}{|\Aut(\boldsymbol{\beta})|}  F^{E,\tw}_{d_E,\boldsymbol{\beta}}(\hbar) \prod_{i}  \V^{S/E}_{\beta_i}(\hbar).
\end{equation*}
\end{prop}
This formula expresses the generating series $\W^{K_S}_\beta$ of Gromov-Witten invariants of the total space of $K_S$ in terms of generating series $\V^{S/E}_{\beta_i}$ of the maximal contact Gromov-Witten invariants of the pair $(S,E)$ and the  series $F^{E,\tw}_{d_E,\boldsymbol{\beta}}$ for a local elliptic curve. See Section \ref{Section:loc/rel} for more details.

Using this equation, one can iteratively solve for $\V^{S/E}$ in terms of $\W^{K_S}$ and $F^{E,\tw}$ (see Corollary \ref{cor:rel/local}). Combing this result with the refined sheaves/Gromov-Witten correspondence leads to a crucial equality (see also equation \eqref{eqn:sheaf/GW}) for Theorem \ref{thm:main0}:

\begin{equation}\label{eqn:crueqn}
\begin{split}
\sum_{k|\beta} \frac{-1}{k(y^{\frac{k}{2}}-y^{-\frac{k}{2}})^2} \left(\frac{\frac{\beta\cdot E}{k}(-1)^{\frac{\beta^2}{k^2}+1}\Omega_{\frac{\beta}{k}}\left(y^{\frac{k}{2}}\right)}{\frac{\left(y^{\frac{k}{2}}\right)^{\frac{\beta\cdot E}{k}}-\left(y^{\frac{k}{2}}\right)^{-\frac{\beta\cdot E}{k}}}{y^{\frac{k}{2}}-y^{-\frac{k}{2}}}}-\sum_{g\geq 0}n_{g,\frac{\beta}{k}}(-1)^g(y^{\frac{k}{2}}-y^{-\frac{k}{2}})^{2g}\right) &\\
=\sum_{T\in \RT_\beta,\atop t^{-1}(0)\neq \emptyset} \Cont_{T}.
\end{split}
\end{equation}
where $n_{g,\beta}$ are Gopakumar-Vafa invariants of $K_S$, and the RHS of the equality \eqref{eqn:crueqn} is a sum over rooted trees such that each contribution $\Cont_{T}$ can be expressed explicitly in terms of generating series of Gromov-Witten invariants of $K_S$ and a local elliptic curve. 

The next crucial step in proving Theorem \ref{thm:main0} is to establish that each 
$$ \Cont_{T}\in\Q[y,y^{-1}]$$
which follows from an explicit formula (see Theorem \ref{thm:keythm} for more details) for stationary Gromov-Witten invariants of local elliptic curves. This formula is established by extending the topological
quantum field theory (TQFT) formalism of \cite{BP08} to include stationary insertions. 

Using induction and assuming Conjecture \ref{conj:rsh/GW}, we derive an even stronger result:
\[\frac{-1}{(y^{\frac{1}{2}}-y^{-\frac{1}{2}})^2} \left(\frac{(\beta\cdot E)(-1)^{\beta^2+1}\Omega_{\beta}\left(y^{\frac{1}{2}}\right)}{{\frac{\left(y^{\frac{1}{2}}\right)^{\beta\cdot E}-\left(y^{\frac{1}{2}}\right)^{-\beta\cdot E}}{y^{\frac{1}{2}}-y^{-\frac{1}{2}}}}}-\sum_{g\geq 0}n_{g,\beta}(-1)^g(y^{\frac{1}{2}}-y^{-\frac{1}{2}})^{2g}\right)\in \Z[y,y^{-1}].\]

\begin{rmk}
There is a quit parallel story in \cite{BW23}. Let $D$ be a smooth rational ample divisor in a smooth projective surface $X$. It was shown in \cite[Theorem 1.7]{BW23} that the generating series of maximal contact Gromov-Witten invariants of the pair $(X,D)$ (with fixed curve class $\beta$) can be divided by the (shifted) Poincar\'e polynomial of projective space $\PP^{D\cdot\beta-1}$. 
Using the quiver/Gromov-Witten correspondence \cite{Bou20, Bou21},
this is equivalent to the fact that the (intersection) Poincar\'e polynomial of certain quiver moduli space can be divided by the Poincar\'e polynomial of $\PP^{D\cdot\beta-1}$. 

In contrast with with the case of $M_{\beta}$, there is a geometric explanation for the divisibility of Poincar\'e polynomial in the quiver case.
Actually, it was shown in \cite[Theorem 1.11]{BW23} that such a quiver moduli space has a small resolution such that after the small resolution, it becomes a projective bundle.
\end{rmk}

\subsection{Explicit calculations} Next, we specialize $S=\PP^2$ and use $\Omega^{\mathbb P^2}_d$ to denote the shifted Poincar\'e polynomial of the moduli space $M_d$ of Gieseker semistable sheaves on $\PP^2$ supported on curves with degree $d\in \Z_{>0}$. The equality \eqref{eqn:crueqn} also introduces a novel approach to computing the Poincar\'e polynomials $\Omega^{\mathbb P^2}_d$ using Gromov-Witten invariants of local $\PP^2$ and stationary Gromov-Witten invariants of a local elliptic curve. In Section \ref{sec:Nume}, we give a detailed computation of $\Omega^{\mathbb P^2}_d$ for $d\leq 6$ using this new method, and provide a list of $\hat{\Omega}^{\mathbb P^2}_d$ for $d\leq 10$ where
\begin{equation}\label{eqn:shiftOmega}
\hat{\Omega}^{\mathbb P^2}_d\coloneqq y^{\frac{(d-1)(d-2)}{2}}\cdot \frac{\Omega^{\mathbb P^2}_d}{{{\frac{\left(y^{\frac{1}{2}}\right)^{3d}-\left(y^{\frac{1}{2}}\right)^{-3d}}{y^{\frac{1}{2}}-y^{-\frac{1}{2}}}}}}=\frac{P_d(y)}{\sum_{i=0}^{3d-1}y^i}\in \Z[y]. 
\end{equation}
In fact, we have calculated $\hat{\Omega}^{\mathbb P^2}_d$ for $d\leq 16$ using Gopakumar-Vafa invariants of local $\PP^2$
in \cite{HKR}. Since the coefficients of $\hat{\Omega}^{\mathbb P^2}_d$ grow rapidly, we limit the list to $d\leq 10$ in this paper. A complete list of $\hat{\Omega}^{\mathbb P^2}_d$ for $d\leq 16$ is available at

\vspace{1em}
\centerline{\url{https://sites.google.com/site/guoshuaimath/poincarepolynomials}}
\vspace{1em}

\noindent
The new calculation matches with the results in the literature \cite{Y14,Y18,Y23,Y24,CC15,CC16,CvGKT}. 

Equality \eqref{eqn:crueqn} also enables us to give a closed formula for the leading Betti numbers of $M_d$ when $d\geq 6$:
\begin{thm}[=Theorem \ref{thm:2nd}]\label{thm:leadingBetti}
For $d\geq 6$, we have
$$
\hat{\Omega}^{\mathbb P^2}_d \bigg|_{y^{\leq 2d-11}} =  \prod_{k>0} \frac{1}{(1-y^{k})(1-y^{k+1})^{2}} \Big( 1 - 3\,y^{d-1}\cdot \frac{ 1+  y + y^2   }{1-y}  \Big) \bigg|_{y^{\leq 2d-11}}.
$$
\end{thm}
Throughout, $|_{y^{\leq m}}$ denotes truncation (keeping terms with $y$-powers $\leq m$). A geometric interpretation of the above formula in terms of the cohomology rings of $M_{d,\chi}$ is discussed in Appendix \ref{subsection:geometry}. Theorem \ref{thm:leadingBetti}, proven in Section \ref{sec:Leading}, can be summarized as follows:

A detailed $y$-degree analysis (see Propositions \ref{prop:degbd1} and \ref{prop:degbd2}) on the RHS of \eqref{eqn:crueqn} reveals that only two rooted trees contribute to the leading Betti numbers. Theorem \ref{thm:leadingBetti} then follows from a closed formula for the leading Gopakumar-Vafa invariants, derived by the first author and Zhou in \cite{GZ15} using the topological vertex formula (see Theorem \ref{thm:gvbd} for more details).

\begin{rmk}
In \cite{Y23,Y24}, Yuan showed that the Betti numbers $b_i(M_{d})$ can be computed from the Betti numbers of Hilbert schemes of points in $\PP^2$ when $i\leq 2d$. Yuan's results can be reformulated as 
$$
\hat{\Omega}^{\mathbb P^2}_d \bigg|_{y^{\leq d}} =  \prod_{k>0} \frac{1}{(1-y^{k})(1-y^{k+1})^{2}} ( 1 - 3\,y^{d-1}-6y^d ) \bigg|_{y^{\leq d}}, \quad d\geq 5
$$
which agrees with Theorem \ref{thm:leadingBetti} and the numerical results in Section \ref{sec:Nume}. Our results extend the computation of leading Betti numbers $b_i(M_d)$ from $i\leq 2d$ to $i\leq 4d-22$.
\end{rmk}

\begin{rmk}\label{rmk:betterbd}
Based on the numerical results in Section \ref{sec:Nume}, we observe that the equality in Theorem \ref{thm:leadingBetti} holds for
degrees up to $2d-5$, rather than $2d-11$. Moreover, a more detailed analysis of the numerical data in Section \ref{sec:Nume} leads us to propose the following conjecture:
\end{rmk}

\begin{conj} \label{conj:2ndappro}
For $d\geq 4$, we have 
$$
\hat{\Omega}^{\mathbb P^2}_d \bigg|_{y^{\leq 3d-10}}  =  \prod_{k>0} \frac{1}{(1-y^{k})(1-y^{k+1})^{2}} \Big( 1 - 3\,y^{d-1}\cdot \frac{ 1+  y + y^2   }{1-y} + 3\, y^{2d-4}\cdot f(y) \Big) \bigg|_{y^{\leq 3d-10}}
$$
where 
$$f(y)=\frac{(1 + y + y^2) (-2 + 2 y + 4 y^2 + 2 y^3 + y^4 + 2 y^5)}{(1 - y) (1 - y^2)}. $$

% $f(y) $ is independent of $d$, with the leading expansion
% $$
% f(y) = -2-2\,y+8\,{y}^{3}+17\,{y}^{4}+30\,{y}^{5}+42\,{y}^{6}+57\,{y}^{7}+69
% \,{y}^{8}+84\,{y}^{9}+96\,{y}^{10}+O( {y}^{11} ).
% $$ 
\end{conj}
A potential resolution of the above conjecture may require a more refined analysis of Gopakumar-Vafa invariants of local $\PP^2$. We leave this for future exploration.

\begin{rmk}
Conjecture \ref{conj:2ndappro} was initially formulated when $f(y)$ was known only up to order 10. The precise closed-form expression for $f(y)$ was proposed by M. Moreira based on our preliminary coefficients and certain geometric considerations. This further inspired Moreira to propose a more general conjecture on higher-range Betti numbers of $M_d$ (Conjecture \ref{conj: higherrange}), as well as another conjecture (Conjecture \ref{conj: higherrangerefined}) involving refinements coming from the perverse/Chern filtration, both presented in Appendix \ref{sec:higherrange}.
\end{rmk}

\subsection{Organization of the paper}
This paper is organized as follows. In Section \ref{Section:loc/rel}, we give a new way to package the all-genus local/relative correspondence. In Section \ref{sec:localcurvesfm}, we give an explicit formula for Gromov-Witten invariants of local elliptic curves with stationary insertions. In Section \ref{sec:pfCGKT}, we give a proof of Theorem \ref{thm:main0}, and we give a detailed $y$-degree analysis for the contributions arising in the graph sum formula established in Section \ref{Section:loc/rel}. In Section \ref{sec:Leading}, we give a proof of Theorem \ref{thm:leadingBetti}. In Section \ref{sec:Nume}, we give a detailed calculation of the shifted Poincar\'e polynomial $\Omega^{\mathbb P^2}_d$ for $d\leq 6$ and provide a list of $\hat{\Omega}^{\mathbb P^2}_d$ up to $d\leq 10$. In Appendix \ref{sec:higherrange}, a more general conjecture concerning the higher range Betti numbers of $M_{d}$ is presented, along with another conjecture that involves refinements from the perverse/Chern filtration.

\subsection{Acknowledgment}
We would like to thank Pierrick Bousseau, Jinwon Choi, Michel van Garrel, Rahul Pandharipande, Junliang Shen for helpful discussions and comments. Special thanks go to Miguel Moreira for contributing the Appendix and for his help in formulating Conjecture \ref{conj:2ndappro}.

The first author is supported by the National Key R\&D Program
of China (No. 2023YFA1009802)
and NSFC 12225101. The second author is supported by the National Key R\&D Program of China (No. 2022YFA1006200) and SwissMAP. The second author is especially grateful to Rahul Pandharipande for hosting an academic visit at ETH Zurich, which greatly facilitated the completion of this work.

\section{All-genus local/relative correspondence}\label{Section:loc/rel}
In this section, we 
give a new way to package the all-genus local/relative correspondence established in \cite{BFGW}.

Let $E$ be a smooth effective anticanonical divisor on a del Pezzo surface $S$. Let $\bM_{g}(K_S,\beta)$ be the moduli space of stable maps to the total space of $K_S$ with genus $g$, curve class $\beta$. We define the local Gromov-Witten invariants
\[ N_{g,\beta}^{K_S} \coloneqq \int_{[\bM_g(K_S,\beta)]^\vir} 1 \,.\]
Let $\bM_{g}(S/E,\beta)$ be the moduli space of genus $g$ relative stable maps of class $\beta$ to $(S,E)$ with only one contact condition of maximal tangency along $E$. It has virtual dimension $g$ and we define the maximal contact Gromov-Witten invariants of the pair $(S,E)$
\[ N_{g,\beta}^{S/E}\coloneqq \int_{[\bM_g(S/E,\beta)]^\vir}
(-1)^g \lambda_g \,.\]
Let $N_{E/S}$ be the normal bundle to $E$ in $S$. We consider the rank $2$ vector bundle
$N \coloneqq N_{E/S}\oplus N_{E/S}^{\vee}$ over $E$
and the anti-diagonal scaling action of $\GM$ on $N$ with weight $1$ on $N_{E/S}$ and weight $-1$ on $N_{E/S}^{\vee}$. 
We denote by $t$ the corresponding equivariant parameter. Let 
$\bM_{g,n}(E,d_E)$ be the moduli space of $n$-pointed, genus $g$ stable maps of degree $d_E$ to the elliptic curve $E$. We define the local invariants of $E$ to be
\begin{equation*}
N_{g,d_E}^{\Etw}(\bd)
\coloneqq \int_{[\bM_{g,n}(E,d_E)]^{\vir}} \left( \prod_{j=1}^n \frac{t \ev_j^*\omega}{t-d_j \psi_j} \right) e_{\GM}(-R^{\bullet}\pi_*f^{*}N)
\end{equation*}
where $\bd=(d_1,d_2,\cdots,d_n)$ is a tuple of positive integers, $\psi_j$ are cotangent line classes, $\omega$ is the point class of $E$, $\pi$ is the universal domain curve and $f$ is the universal stable map.

The local invariants
$N_{g,\beta}^{K_S}$ and relative invariants $N_{g,\beta}^{S/E}$ can be related via the following all-genus local/relative correspondence \cite{BFGW}:
\begin{equation}\label{eqn:loc/rel_origin}
\begin{split}
N_{g,\beta}^{K_{S}}
=& \frac{(-1)^{\beta \cdot E-1}}{\beta
\cdot E} N_{g,\beta}^{S/E}+ \\
&\sum_{n\geq 0}\sum_{\substack{g=g_0+g_1+\dots+g_n\\
\beta=d_E[E]+\beta_1+\dots +\beta_n \\
d_E\geq 0,\,\beta_j\cdot E>0}}  
\frac{N_{g_0,d_E}^{\Etw}(\boldsymbol{\beta} \cdot E)}{|\Aut (\boldsymbol{\beta},\bg)|}
\prod_{j=1}^n \left((-1)^{\beta_j \cdot E}
(\beta_j \cdot E)
N_{g_j,\beta_j}^{S/E}\right).
\end{split}
\end{equation}
Here 
\[\boldsymbol{\beta}\coloneqq(\beta_1,\dots,\beta_n )\,, \qquad \boldsymbol{\beta} \cdot E \coloneqq(\beta_1\cdot E,\dots,\beta_n \cdot E) \,,\]
\[\bg\coloneqq (g_1,\cdots,g_n)\,,\quad  \frac{1}{|\Aut (\boldsymbol{\beta},\bg)|}\coloneqq \frac{1}{|\Aut ((\beta_1,g_1),\cdots,(\beta_n,g_n))|},\]
and we simply set $N_{g_0,d_E}^{\Etw}(\boldsymbol{\beta} \cdot E)$ to be zero if the corresponding moduli is unstable.

\subsection{A new way to package the all-genus local/relative correspondence}
We introduce the following generating series by summing over genera:
$$
\W_{\beta}^{K_S}(\hbar)\coloneqq \sum_{g \geq 0} N_{g,\beta}^{K_{S}} \hbar^{2g-2}, \qquad 
 \V_{\beta}^{S/E}(\hbar)\coloneqq {(-1)^{\beta\cdot E-1}}  \sum_{g \geq 0}  \frac{N_{g,\beta}^{S/E}}{\beta\cdot E} \hbar^{2g-1},
$$
$$
F^{E,\tw}_{d_E,\boldsymbol{\beta}}(\hbar)\coloneqq \sum_{g \geq 0}\left( \prod_{i=1}^n(-(\beta_i\cdot E)^2)\right)\cdot N_{g,d_E}^{\Etw}(\boldsymbol{\beta} \cdot E) \hbar^{2g-2+n}.
$$
Then we can package the all-genus local/relative correspondence \eqref{eqn:loc/rel_origin} as 
\begin{prop}\label{prop:local/rel}
For a fixed curve class $\beta>0$, we have
\begin{align*}
\W^{K_S}_\beta(\hbar)   =  \ & \sum_{0\leq d_E[E]\leq \beta}  \sum_{\boldsymbol{\beta}=(\beta_1,\cdots,\beta_n)\atop |\boldsymbol{\beta}|+d_E[E] =\beta}  \frac{1}{|\Aut(\boldsymbol{\beta})|}  F^{E,\tw}_{d_E,\boldsymbol{\beta}}(\hbar) \prod_{i}  \V^{S/E}_{\beta_i}(\hbar)
\\  =  \ &    \    \frac{(\beta\cdot E) \V^{S/E}_{\beta}(\hbar)}{ 2 \sin \frac{(\beta\cdot E) \hbar}{2} } 
    + \sum_{0< d_E[E]\leq \beta}  \sum_{\boldsymbol{\beta}=(\beta_1,\cdots,\beta_n)\atop |\boldsymbol{\beta}|+d_E[E] =\beta} \frac{1}{|\Aut(\boldsymbol{\beta})|}  F^{E,\tw}_{d_E,\boldsymbol{\beta}}(\hbar) \prod_{i}  \V^{S/E}_{\beta_i}(\hbar).
\end{align*}
\end{prop}
Here $[E]$ is the curve class associated to $E$,
\[ |\boldsymbol{\beta}|\coloneqq \sum_{i=1}^n \beta_i,\qquad \frac{1}{|\Aut (\boldsymbol{\beta})|}\coloneqq \frac{1}{|\Aut (\beta_1,\cdots,\beta_n)|},\]
and the second equality follows from the fact that $$F^{E,\tw}_{d_E=0,\boldsymbol{\beta}=(\beta_1,\dots,\beta_n)}(\hbar)=
\begin{cases}
\frac{\beta_1\cdot E}{2 \sin \frac{(\beta_1\cdot E) \hbar}{2} }, & \text{ if } n=1,\\
0, & \text{ if } n>1.
\end{cases}
$$

\subsection{A graph sum formula}  \label{sec:graphsum}
Using the equation in Proposition \ref{prop:local/rel},
one can iteratively 
solve for $\V^{S/E}$ in terms of $\W^{K_S}$ and $F^{E,\tw}$ as follows.

Let $T$ be a rooted tree\footnote{A rooted tree is a tree with a special vertex labeled as the "root" the of tree. It  serves as a point of reference for other vertices in the tree.}  together with the following assignments:

\begin{enumerate}
\item[(i)] $t:V_T\longrightarrow \{0,1\}$, where we use $V_T$ to denote the set of vertices. 
\item[(ii)] $d^0:t^{-1}(0)\longrightarrow \Z_{>0}$, $d^1:t^{-1}(1)\longrightarrow H_2(S,\Z)$.
\item[(iii)] $d_2:E_T\longrightarrow H_2(S,\Z)$, where we use $E_T$ to denote the set of edges.
\end{enumerate}
The first assignment $t$ gives a labeling of vertices. A vertex labeled by $0$ will be represented by a circle and a vertex labeled by $1$ will be represented by a square. The rest assignments give another labeling of vertices and edges by curve classes of $S$. We require these labelings to satisfy:
\begin{enumerate}
\item[(1)] If the number of vertices $|V_T|>1$, then vertices labeled by $1$ are always leaves\footnote{A leaf in a rooted tree is any vertex having no children. So the rooted vertex is always not a leaf if the number of vertices is greater than $1$.} of the rooted tree $T$. 
\item[(2)] The images of $d^1$ and $d_2$ are nonzero effective curve classes.
\item[(3)] For any vertex $v_i$ labeled by $0$ which is not the rooted vertex, we use $e^i_0,e^i_1,\cdots,e^i_{n_i}$ to denote the edges incident to $v_i$ where $e_0^i$ is the unique edge belongs to the path connecting $v_i$ to the rooted vertex. We require
\[d_2(e_0^i)=d^0(v_i)[E]+\sum_{k=1}^{n_i}d_2(e_k^i).\]
If $|V_T|>1$, then for each vertex $v_i$ labeled by $1$, there is a unique edge $e_0^i$ incident to $v_i$. We further require
$$
d_2(e_0^i)=d^1(v_i).
$$
\end{enumerate}
Given such a $T\in \RT$, we define the degree of $T$ to be 
\[\sum_{v\in t^{-1}(1)}d^1(v)+\sum_{v\in t^{-1}(0)} d^0(v)[E].\]
Two labeled rooted trees are called isomorphic if there exists an isomorphism\footnote{An isomorphism between rooted trees always maps rooted vertex to rooted vertex.} between the underlying rooted trees which commutes with all the labelings. We use $\RT$ to denote the isomorphism classes of all such labeled rooted trees. We further use $\RT_\beta$ to denote the labeled rooted trees in $\RT$ with fixed curve class $\beta$. 

Let us give a concrete example of a labeled rooted tree in the case $S=\PP^2$:

\vspace{1em}
\tikzset{every picture/.style={line width=0.75pt}} %set default line width to 0.75pt        

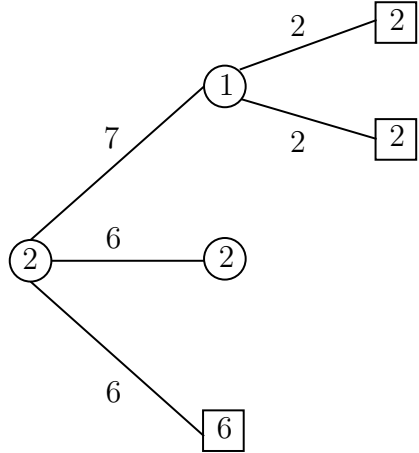
\begin{center}
\begin{tikzpicture}[x=0.75pt,y=0.75pt,yscale=-1,xscale=1]

%Shape: Circle [id:dp5446169583573187] 
\draw   (100,159.5) .. controls (100,153.7) and (104.7,149) .. (110.5,149) .. controls (116.3,149) and (121,153.7) .. (121,159.5) .. controls (121,165.3) and (116.3,170) .. (110.5,170) .. controls (104.7,170) and (100,165.3) .. (100,159.5) -- cycle ;

%Shape: Circle [id:dp6771752395433279] 
\draw   (198,71.5) .. controls (198,65.7) and (202.7,61) .. (208.5,61) .. controls (214.3,61) and (219,65.7) .. (219,71.5) .. controls (219,77.3) and (214.3,82) .. (208.5,82) .. controls (202.7,82) and (198,77.3) .. (198,71.5) -- cycle ;
%Shape: Circle [id:dp636439526625897] 
\draw   (198,158.5) .. controls (198,152.7) and (202.7,148) .. (208.5,148) .. controls (214.3,148) and (219,152.7) .. (219,158.5) .. controls (219,164.3) and (214.3,169) .. (208.5,169) .. controls (202.7,169) and (198,164.3) .. (198,158.5) -- cycle ;
%Shape: Square [id:dp05663682075637111] 
\draw   (197.5,235) -- (218,235) -- (218,255.5) -- (197.5,255.5) -- cycle ;

%Shape: Square [id:dp4436975094931739] 
\draw   (284.5,29) -- (305,29) -- (305,49.5) -- (284.5,49.5) -- cycle ;

%Shape: Square [id:dp050036265286199644] 
\draw   (284.5,88) -- (305,88) -- (305,108.5) -- (284.5,108.5) -- cycle ;

%Shape: Square [id:dp01023578412217041] 
% \draw   (283.5,149) -- (304,149) -- (304,169.5) -- (283.5,169.5) -- cycle ;
%Straight Lines [id:da7070219457632385] 
\draw    (121,159.5) -- (198,159.5) ;
%Straight Lines [id:da891242303205514] 
\draw    (110.5,149) -- (198,71.5) ;
%Straight Lines [id:da05218526421629699] 
\draw    (110.5,170) -- (198,248) ;
%Straight Lines [id:da9517490270948519] 
% \draw    (219,158.5) -- (284,159) ;
%Straight Lines [id:da11829825487028334] 
\draw    (217,78) -- (285,98) ;
%Straight Lines [id:da6597740191371922] 
\draw    (216,63) -- (285,38) ;

% Text Node
\draw (105,152) node [anchor=north west][inner sep=0.75pt]   [align=left] {2};
% Text Node
\draw (204,64) node [anchor=north west][inner sep=0.75pt]   [align=left] {1};
% Text Node
\draw (204,151) node [anchor=north west][inner sep=0.75pt]   [align=left] {2};
% Text Node
\draw (203,237) node [anchor=north west][inner sep=0.75pt]   [align=left] {6};
% Text Node
\draw (290,31) node [anchor=north west][inner sep=0.75pt]   [align=left] {2};
% Text Node
\draw (290,90) node [anchor=north west][inner sep=0.75pt]   [align=left] {2};
% Text Node
% \draw (289,151) node [anchor=north west][inner sep=0.75pt]   [align=left] {1};
% Text Node
\draw (240,34) node [anchor=north west][inner sep=0.75pt]   [align=left] {2};
% Text Node
\draw (240,93) node [anchor=north west][inner sep=0.75pt]   [align=left] {2};
% Text Node
\draw (146,91) node [anchor=north west][inner sep=0.75pt]   [align=left] {7};
% Text Node
% \draw (241,142) node [anchor=north west][inner sep=0.75pt]   [align=left] {1};
% Text Node
\draw (147,141) node [anchor=north west][inner sep=0.75pt]   [align=left] {6};
% Text Node
\draw (147,219) node [anchor=north west][inner sep=0.75pt]   [align=left] {6};

\end{tikzpicture}
\end{center}
\vspace{1em}
\noindent
Since $S=\PP^2$, curve classes can be represented by integers.
The assigned integers are put in the vertices or along the edges.
We always put the rooted vertex on the leftmost. The degree of the above tree is $25$. 

We can express $\V^{S/E}_\beta$ as a sum over labeled rooted trees in $\RT_\beta$ by iteratively solving the equation in Proposition \ref{prop:local/rel}:
\begin{cor}\label{cor:rel/local}
For a fixed $\beta>0$, we have
\begin{align*}
\V^{S/E}_\beta = \ & \frac{2\sin\frac{(\beta\cdot E)\hbar}{2}}{\beta\cdot E}\sum_{T\in \RT_\beta}\Cont_{T}\\
= \ & \frac{2\sin{\frac{(\beta\cdot E)\hbar}{2}}}{\beta\cdot E}\big(\W^{K_S}_\beta+\sum_{T\in \RT_\beta,\atop t^{-1}(0)\neq \emptyset} \Cont_{T}\big).
\end{align*}
% \[\V_d=\frac{2\sin{\frac{3d\hbar}{2}}}{3d}\W_d+\sum_{T\in \RT_d} \text{Cont}_{T}\]
Here the contribution of each rooted tree $T$ in $\RT_\beta$ is given by
\begin{equation}\label{eq:cont}
\begin{split}
\Cont_{T}= & \frac{1}{|\Aut(T)|}
\prod_{e\in E_T}\frac{2\sin\frac{(d_2(e)\cdot E)\hbar}{2}}{d_2(e)\cdot E} \cdot\\
&\prod_{v_i\in t^{-1}(0)}\left(-F^{E,\tw}_{d^0(v_i),\{d_2(e_1^i),d_2(e_2^i),\cdots,d_2(e_{n_i}^i)\}}\right)\prod_{v_i\in t^{-1}(1)}\W^{K_S}_{d^1(v_i)}
\end{split}
\end{equation}
where $|\Aut(
T)|$ stands for the number of automorphism of $T$ as a labeled rooted tree, $e^i_1,\cdots,e^i_{n_i}$ are the edges incident to $v_i$
which do not belong to the path connecting $v_i$ and the rooted vertex.
\end{cor}
The second equality in Corollary \ref{cor:rel/local} follows from the fact that there is only one rooted tree $T\in\RT_\beta$ such that $t^{-1}(0)=\emptyset$ and the corresponding contribution is
\[\frac{2\sin{\frac{(\beta\cdot E)\hbar}{2}}}{\beta\cdot E}\W^{K_S}_\beta.\]
For the labeled rooted tree $T$ given above, the contribution $\Cont_{T}$ is
\[-\frac{1}{2}\frac{2\sin{\frac{21\hbar}{2}}}{21}\frac{2\sin{\frac{18\hbar}{2}}}{18}\frac{2\sin{\frac{18\hbar}{2}}}{18}\frac{2\sin{\frac{6\hbar}{2}}}{6}\frac{2\sin{\frac{6\hbar}{2}}}{6}F^{E,\tw}_{2,\{7,6,6\}}
F^{E,\tw}_{2,\emptyset}F^{E,\tw}_{1,\{2,2\}}\W^{K_{\PP^2}}_2\W^{K_{\PP^2}}_2\W^{K_{\PP^2}}_6.\]

\section{Gromov-Witten invariants of local elliptic curves}\label{sec:localcurvesfm}
In this section, we give an explicit formula for $F^{E,\tw,\bullet}_{d_E,\boldsymbol{\beta}}$ by extending the topological quantum field theory (TQFT) formalism of Bryan and Pandharipande \cite{BP08} to include stationary insertions. The superscript 
$\bullet$ in $F^{E,\tw,\bullet}_{d_E,\boldsymbol{\beta}}$ indicates that we consider possibly disconnected invariants\footnote{Following \cite{BP08}, we require the degree of each connected domain curve to be nonzero.}.

Let $\rho=(\rho_1\geq \rho_2\geq \cdots)$ be a partition of $d_E$. It corresponds to a Young diagram. We use $c_{\rho}$ to denote the total sum of the contents of all its boxes where the content of a box at the $(i,j)$-position is $j-i$. We further set $k$ to be the degree of the normal bundle $N_{E/S}$ or equivalently $k=K_S^2$, and set
\begin{equation}\label{eqn:tildeE}
\tilde{e}(\rho,z) \coloneqq \sum_{i=1}^{\infty}(e^{z (\rho_i-i+\frac{1}{2})}-e^{z(-i+\frac{1}{2})}), 
\end{equation}
where we extend $\rho$ to an infinite sequence by appending zeros.

\begin{thm}\label{thm:keythm}
\[F^{E,\tw,\bullet}_{d_E,\boldsymbol{\beta}}(\hbar)=\sum_{|\rho|=d_E}  (-1)^{kd_E }e^{k c_\rho \sqrt{-1}   \hbar} \prod_{i=1}^n (-\beta_i\cdot E)\frac{\tilde{e}(\rho,(\beta_i\cdot E)\sqrt{-1}\hbar)}{\sqrt{-1}}\]
where $|\rho|=\sum_i\rho_i$.
\end{thm}
\begin{proof}
When there are no interior markings, i.e., $\beta=\emptyset$, Theorem \ref{thm:keythm} specializes to 
\begin{equation}\label{eqn:empty}
F^{E,\tw,\bullet}_{d_E,\emptyset}(\hbar)=\sum_{|\rho|=d_E}  (-1)^{kd_E }e^{ k  c_\rho\sqrt{-1}  \hbar} 
\end{equation}
which is simply Corollary 7.4 of \cite{BP08}. It was proven by the 
TQFT formalism of \cite{BP08}. When there are interior markings, we need to extend the TQFT formalism to include stationary insertions. The key 
is to compute the following level $(0,0)$ series with one interior marking\footnote{The level series were introduced in \cite{BP08} when there are no interior markings. We borrow the notation from \cite{BP08} for level series. But our level series contain stationary insertions. So we add a tilde to indicate the difference.}:
\[\widetilde{\GW}(0|0,0)^{\alpha}(\beta_i)\coloneqq \hbar^{d_E+l(\alpha)}(-t^2)^{l(\alpha)}\mathfrak{z}(\alpha)\tilde{Z}(N)_{\alpha}(\beta_i)\]
where $\alpha$ is a partition of $d_E$, $l(\alpha)$ denotes its length, and
\[\mathfrak{z}(\alpha)=\prod_i m_i(\alpha)!i^{m_i(\alpha)}\]
with $m_i(\alpha)$ being the multiplicity of $i$ in $\alpha$. The term $\tilde{Z}(N)_{\alpha}(\beta_i)$ is defined as
\[\tilde{Z}(N)_{\alpha}(\beta_i)\coloneqq \sum_{g\in \Z}\hbar^{2g-1}(-(\beta_i\cdot E)^2)\int_{[\bM^{\bullet}_{g,1}(\PP^1/0,\alpha)]^{\vir}}\frac{t \ev_1^*(\omega)}{t-(\beta_i\cdot E) \psi_1}e_{\GM}(-R^{\bullet}\pi_* f^*(\mathcal{O}\oplus \mathcal{O}))\]
where $\bM^{\bullet}_{g,1}(\PP^1/0,\alpha)$ is the possibly disconnected moduli space of genus $g$ relative stable maps to $(\PP^1,0)$ with one interior marking and tangency condition $\alpha$,
$\GM$ acts anti-diagonal on the direct sum of two trivial line bundles, and $t$ is the equivariant parameter. By Mumford's relation and dimension constraint, we have 
\[\widetilde{\GW}(0|0,0)^{\alpha}(\beta_i)=\sum_{g\in \Z}\hbar^{k+1}(-1)^{l(\alpha)+g}t^{l(\alpha)-d_E}\mathfrak{z}(\alpha)\int_{[\bM^{\bullet}_{g,1}(\PP^1/0,\alpha)]^{\vir}}(\beta_i\cdot E)^{k+2}\psi_1^k \ev_1^*(\omega).\]
with
\begin{equation}\label{eqn:g2k}
k=d_E+2g-2+l(\alpha).
\end{equation}
According to \cite[Lemma 7.5]{BP08}, up to some factors of $t$, the Frobenius algebra associated to the TQFT in the anti-diagonal case is isomorphic to the center of the group algebra of the symmetric group. The idempotent basis are indexed by partitions of $d_E$ (equation (20) and (21) of \cite{BP08}):
\[v_{\rho}=\frac{\dim \rho}{(d_E)!}\sum_{\alpha}(\sqrt{-1}t)^{l(\alpha)-d_E}\chi^{\rho}_{\alpha}e_{\alpha}\]
or equivalently, we have
\[e_{\alpha}=(\sqrt{-1}t)^{d_E-l(\alpha)}\sum_{\rho}\frac{(d_E)!}{\dim \rho}\frac{\chi^{\rho}_{\alpha}}{\mathfrak{z}(\alpha)}v_{\rho}.\]
So if we write $\sum_{\alpha}\widetilde{\GW}(0|0,0)^{\alpha}(\beta_i)e_{\alpha}$ in terms of the idempotent basis $v_{\rho}$, the coefficient before $v_{\rho}$ is simply
\begin{equation}\label{eqn:bigsum}
\gamma_{\rho}(\beta_i)=\sum_{k\geq 0} (\sqrt{-1})^{k+2}\hbar^{k+1}(\beta_i\cdot E)^{k+2}\sum_{|\alpha|=d_E}\frac{(d_E)!}{\dim \rho}\chi^{\rho}_{\alpha}\int_{[\bM^{\bullet}_{g,1}(\PP^1/0,\alpha)]^{\vir}}\psi_1^k \ev_1^*(\omega).
\end{equation}
We are ready to extend the TQFT formalism of \cite{BP08} to include stationary insertions:
\[F^{E,\tw,\bullet}_{d_E,\boldsymbol{\beta}}=\sum_{|\rho|=d_E}  (-1)^{kd_E }e^{kc_\rho \sqrt{-1}    \hbar} \prod_{i=1}^n \gamma_{\rho}(\beta_i)\]
which can be proved using the degenerating formula as in \cite{BP08}.

We are left to determine $\gamma_{\rho}(\beta_i)$. By the Gromov-Witten/Hurwitz correspondence \cite{OP1}, we know that 
\[\sum_{|\alpha|=d_E}\frac{(d_E)!}{\dim \rho}\chi^{\rho}_{\alpha}\int_{[\bM^{\bullet}_{g,1}(\PP^1/0,\alpha)]^{\vir}}\tau_k(\omega)=[z^{k+1}]\tilde{e}(\rho,z)\]
where $[z^{k+1}]$ means we take the $z^{k+1}$-coefficient of $\tilde{e}(\rho,z)$. Note that we require the degree of each connected domain curve to be nonzero, so we need to remove the degree zero contributions from the Gromov-Witten/Hurwitz correspondence. After plugging into the \eqref{eqn:bigsum}, it becomes
\[\sum_{k\geq 0} (\sqrt{-1})^{k+2}\hbar^{k+1}(\beta_i\cdot E)^{k+2}[z^{k+1}]\tilde{e}(\rho,z)=(-\beta_i\cdot E)\frac{\tilde{e}(\rho,(\beta_i\cdot E)\sqrt{-1}\hbar)}{\sqrt{-1}}.\]
From the above discussion, we may conclude that 
\[F^{E,\tw,\bullet}_{d_E,\boldsymbol{\beta}}(\hbar)=\sum_{|\rho|=d_E}  (-1)^{kd_E }e^{kc_\rho \sqrt{-1}\hbar} \prod_{i=1}^n (-\beta_i\cdot E)\frac{\tilde{e}(\rho,(\beta_i\cdot E)\sqrt{-1}\hbar)}{\sqrt{-1}}.\]
\end{proof}

\begin{cor}\label{cor:rat}
Under the change of variables $y=e^{i\hbar}$, we have
\[F^{E,\tw,\bullet}_{d_E,\boldsymbol{\beta}}(\hbar)\in \left(\prod_{i=1}^n\frac{y^{\frac{\beta_i\cdot E}{2}}-y^{-\frac{\beta_i\cdot E}{2}}}{\sqrt{-1}}\right)\mathbb{Q}[y,y^{-1}]\]
where $\boldsymbol{\beta}=(\beta_1,\beta_2,\cdots,\beta_n)$. 
\end{cor}

\begin{proof}
By Theorem \ref{thm:keythm}, we know that 
\[F^{E,\tw,\bullet}_{d_E,\boldsymbol{\beta}}(\hbar)=\sum_{|\rho|=d_E}  (-1)^{kd_E }e^{kc_\rho \sqrt{-1}\hbar} \prod_{i=1}^n (-\beta_i\cdot E)\frac{\tilde{e}(\rho,(\beta_i\cdot E)\sqrt{-1}\hbar)}{\sqrt{-1}}.\]
For each partition $\rho$, we know that 
\[ (-1)^{kd_E }e^{ k c_\rho \sqrt{-1}  \hbar}\in \Q[y,y^{-1}]\]
after the change of variables $y=e^{i\hbar}$. We only need to show each
\[(-\beta_i\cdot E)\frac{\tilde{e}(\rho,(\beta_i\cdot E)\sqrt{-1}\hbar)}{\sqrt{-1}}\in \left(\frac{y^{\frac{\beta_i\cdot E}{2}}-y^{-\frac{\beta_i\cdot E}{2}}}{\sqrt{-1}}\right)\mathbb{Q}[y,y^{-1}]. \]
Let 
\[S(\rho)=\{\rho_1-\frac{1}{2}>\rho_2-\frac{3}{2}>\cdots>\rho_k-k+\frac{1}{2}>\cdots\}\in \Z+\frac{1}{2}\]
and set $S_{+}(\rho)=S(\rho)\backslash \left(\Z_{\leq 0}-\frac{1}{2}\right)$, $S_{-}(\rho)= \left(\Z_{\leq 0}-\frac{1}{2}\right)\backslash S(\rho)$. We have $|S_{+}(\rho)|=|S_{-}(\rho)|$. So we may assume that 
\[S_{+}(\rho)=\{m_1,m_2,\cdots,m_s\},\quad S_{-}(\rho)=\{n_1,n_2,\cdots,n_s\} \]
where $m_i,n_i$ are half-integers. Then we may express $\tilde{e}(\rho,(\beta_i\cdot E)\sqrt{-1} \hbar)$ as
\[\tilde{e}(\rho,(\beta_i\cdot E)\sqrt{-1}\hbar)=\sum_{k=1}^s(e^{(\beta_i\cdot E)m_k\sqrt{-1}\hbar}-e^{(\beta_i\cdot E)n_k\sqrt{-1}\hbar }).\]
Each summand 
\begin{align*}
&e^{(\beta_i\cdot E)m_k\sqrt{-1}\hbar}-e^{(\beta_i\cdot E)n_k\sqrt{-1}\hbar }\\
=&(e^{\frac{\beta_i\cdot E}{2}\sqrt{-1}\hbar }-e^{-\frac{\beta_i\cdot E}{2}\sqrt{-1}\hbar})\left(\sum_{l=n_k+1/2}^{m_k-1/2}e^{(\beta_i\cdot E)l\sqrt{-1}\hbar }\right)\in \left(\frac{y^{\frac{\beta_i\cdot E}{2}}-y^{-\frac{\beta_i\cdot E}{2}}}{\sqrt{-1}}\right)\mathbb{Q}[y,y^{-1}]
\end{align*}
after the change of variables $y=e^{i\hbar}$. It follows that 
\[(-\beta_i\cdot E)\frac{\tilde{e}(\rho,(\beta_i\cdot E)\sqrt{-1}\hbar)}{\sqrt{-1}}\in \left(\frac{y^{\frac{\beta_i\cdot E}{2}}-y^{-\frac{\beta_i\cdot E}{2}}}{\sqrt{-1}}\right)\mathbb{Q}[y,y^{-1}]. \]
\end{proof}

Using Corollary \ref{cor:rat}, we can deduce the following lemma which will play an important role in the proof of Theorem \ref{thm:main0}.
\begin{lem}\label{lem:key}
For each rooted tree $T\in\RT_\beta$ such that $t^{-1}(0)\neq \emptyset$, we have 
\[\Cont_T\in \Q[y,y^{-1}]\]
after the change of variables $y=e^{i\hbar}$.
\end{lem}
\begin{proof}
Recall that 
\[
\Cont_{T}=\frac{1}{|\Aut(T)|}
\prod_{e\in E_T}\frac{2\sin\frac{(d_2(e)\cdot E)\hbar}{2}}{d_2(e)\cdot E}\prod_{v_i\in t^{-1}(0)}\left(-F^{E,\tw}_{d^0(v_i),\{d_2(e_1^i),d_2(e_2^i),\cdots,d_2(e_{n_i}^i)\}}\right)\prod_{v_i\in t^{-1}(1)}\W^{K_S}_{d^1(v_i)}.
\]
By Corollary \ref{cor:rat}, we know that
\[F^{E,\tw,\bullet}_{d^0(v_i),\{d_2(e_1^i),d_2(e_2^i),\cdots,d_2(e_{n_i}^i)\}} \in \left(\prod_{k=1}^{n_i}\frac{y^{\frac{d_2(e_{k}^i)\cdot E}{2}}-y^{-\frac{d_2(e_{k}^i)\cdot E}{2}}}{\sqrt{-1}}\right)\mathbb{Q}[y,y^{-1}].\]
Since connected invariants can be determined from the disconnected invariants by inclusion and exclusion, we still have
\[F^{E,\tw}_{d^0(v_i),\{d_2(e_1^i),d_2(e_2^i),\cdots,d_2(e_{n_i}^i)\}} \in \left(\prod_{k=1}^{n_i}\frac{y^{\frac{d_2(e_{k}^i)\cdot E}{2}}-y^{-\frac{d_2(e_{k}^i)\cdot E}{2}}}{\sqrt{-1}}\right)\mathbb{Q}[y,y^{-1}]\]
for connected invariants. Combining with the fact that
\[2\sin\frac{(d_2(e)\cdot E)\hbar}{2}=\frac{y^{\frac{d_2(e)\cdot E}{2}}-y^{
\frac{-d_2(e)\cdot E}{2}}}{\sqrt{-1}},\quad y=e^{i\hbar}.\]
We may deduce that 
\begin{equation}\label{eqn:midcont}
\Cont_{T}\in\left(\prod_{e\in E}(y^{\frac{d_2(e)\cdot E}{2}}-y^{
\frac{-d_2(e)\cdot E}{2}})^2\prod_{v_i\in t^{-1}(1)}\W_{d^1(v_i)}^{K_S}\right)\Q[y,y^{-1}].
\end{equation}
If $t^{-1}(1)$ is empty, then we already see that 
\begin{equation*}
\Cont_{T}\in\Q[y,y^{-1}].
\end{equation*}
If $t^{-1}(1)$ is not empty, then by definition for each $v_i\in t^{-1}(1)$, there is a unique edge $e_0^i\in E$ incident to $v_i$. Then considering the factor
\[(y^{\frac{d_2(e_0^i)\cdot E}{2}}-y^{
\frac{-d_2(e_0^i)\cdot E}{2}})^2\W_{d^1(v_i)}^{K_S}\]
in \eqref{eqn:midcont}.
We can write $\W_{d^1(v_i)}^{K_S}$ in terms of  Gopakumar-Vafa invariants $n_{g,\beta}$ of $K_S$:
\begin{align*}
\W_{d^1(v_i)}^{K_S} &=\sum_{k|d^1(v_i)} \frac{1}{k(2\sin(k\hbar/2))^2}\sum_{g\geq 0}n_{g,d^1(v_i)/k}\left(2\sin(\frac{k\hbar}{2})\right)^{2g}\\
&=\sum_{k|d^1(v_i)} \frac{-1}{k(y^{\frac{k}{2}}-y^{-\frac{k}{2}})^2}\sum_{g\geq 0}n_{g,d^1(v_i)/k}(-1)^g(y^{\frac{k}{2}}-y^{-\frac{k}{2}})^{2g},\quad y=e^{i\hbar}.
\end{align*}
By requirement (3) in the definition of a labeled rooted tree, $d^1(v_i)=d_2(e_0^i)$, so:
\[(y^{\frac{d_2(e_0^i)\cdot E}{2}}-y^{
\frac{-d_2(e_0^i)\cdot E}{2}})^2\W_{d^1(v_i)}^{K_S}\in \Q[y,y^{-1}].\]
Hence,
\begin{equation*}
\Cont_{T}\in\Q[y,y^{-1}].
\end{equation*}
\end{proof}

\section{Applications of the all-genus local/relative correspondence} \label{sec:pfCGKT}
Recall that in Section \ref{Section:loc/rel}, we give a new way to package the all-genus local/relative correspondence. Here and in the next section, we present two applications of this correspondence.

\subsection{Proof of Theorem \ref{thm:main0}}
The first application establishes that Choi-van Garrel-Katz-Takahashi's conjecture follows from Bousseau's conjectural refined sheaves/Gromov-Witten correspondence. Since the refined sheaves/Gromov-Witten correspondence is known for $S=\mathbb{P}^2$, our result proves Choi-van Garrel-Katz-Takahashi's conjecture in this case.

To begin, we reformulate Choi-van Garrel-Katz-Takahashi's conjecture as follows:

\begin{conj}[=Conjecture \ref{conj:main0}]\label{conj:main}
$\Omega_{\beta}\left(y^{\frac{1}{2}}\right)$ can be divided by the shifted Poincar\'e polynomial of $\PP^{\beta\cdot E-1}$, i.e.,
\[\frac{\Omega_{\beta}\left(y^{\frac{1}{2}}\right)}{{\frac{\left(y^{\frac{1}{2}}\right)^{\beta\cdot E}-\left(y^{\frac{1}{2}}\right)^{-\beta\cdot E}}{y^{\frac{1}{2}}-y^{-\frac{1}{2}}}}}\]
is a palindromic Laurent polynomial of $y$.
\end{conj}

\begin{thm}[=Theorem \ref{thm:main0}]\label{thm:main}
Assuming Conjecture \ref{conj:rsh/GW} holds, then Conjecture \ref{conj:main} holds. In particular, Conjectures  \ref{conj:main} holds when $S=\PP^2$.
\end{thm}
\begin{proof}[Proof of Theorem \ref{thm:main}]
% Notice that 
% \[y^{-\frac{1}{2}(\beta\cdot E-1)}\sum_{j=0}^{\beta\cdot E-1}  y^j=\frac{\left(y^{\frac{1}{2}}\right)^{\beta\cdot E}-\left(y^{\frac{1}{2}}\right)^{-\beta\cdot E}}{y^{\frac{1}{2}}-y^{-\frac{1}{2}}}.\]
% So we only need to show that 
% \[\frac{\Omega_{\beta}\left(y^{\frac{1}{2}}\right)}{{\frac{\left(y^{\frac{1}{2}}\right)^{\beta\cdot E}-\left(y^{\frac{1}{2}}\right)^{-\beta\cdot E}}{y^{\frac{1}{2}}-y^{-\frac{1}{2}}}}}\in \Z[y,y^{-1}].\]
Recall by Corollary \ref{cor:rel/local} of the local/relative correspondence,    
\[\V^{S/E}_\beta=\frac{2\sin{\frac{(\beta\cdot E)\hbar}{2}}}{\beta\cdot E}\left(\W^{K_S}_\beta+\sum_{T\in \RT_\beta,\atop t^{-1}(0)\neq \emptyset} \Cont_{T}\right).\]
Assuming Conjecture \ref{conj:rsh/GW} holds, we have
\[\V^{S/E}_\beta=\sqrt{-1} \sum_{k|\beta} \frac{1}{k^2} \frac{(-1)^{\frac{\beta^2}{k^2}+1}\Omega_{\frac{\beta}{k}}\left(y^{\frac{k}{2}}\right)}{y^{\frac{k}{2}}-y^{-\frac{k}{2}}}. \]
Combining them, we get
\[\sum_{k|\beta} \frac{-1}{k(y^{\frac{k}{2}}-y^{-\frac{k}{2}})^2} \frac{\frac{\beta\cdot E}{k}(-1)^{\frac{\beta^2}{k^2}+1}\Omega_{\frac{\beta}{k}}\left(y^{\frac{k}{2}}\right)}{\frac{\left(y^{\frac{k}{2}}\right)^{\frac{\beta\cdot E}{k}}-\left(y^{\frac{k}{2}}\right)^{-\frac{\beta\cdot E}{k}}}{y^{\frac{k}{2}}-y^{-\frac{k}{2}}}}=\W_\beta^{K_S}+\sum_{T\in \RT_\beta,\atop t^{-1}(0)\neq \emptyset} \Cont_{T}.\]
We can write $\W_\beta^{K_S}$ in terms of Gopakumar-Vafa invariants $n_{g,\beta}$ of $K_S$:
\begin{align*}
\W_\beta^{K_S} &=\sum_{k|\beta} \frac{1}{k(2\sin(k\hbar/2))^2}\sum_{g\geq 0}n_{g,\frac{\beta}{k}}\left(2\sin(\frac{k\hbar}{2})\right)^{2g}\\
&=\sum_{k|\beta} \frac{-1}{k(y^{\frac{k}{2}}-y^{-\frac{k}{2}})^2}\sum_{g\geq 0}n_{g,\frac{\beta}{k}}(-1)^g(y^{\frac{k}{2}}-y^{-\frac{k}{2}})^{2g}.
\end{align*}
So after plugging into the previous equality, we get 
\begin{equation}\label{eqn:sheaf/GW}
\begin{split}
\sum_{k|\beta} \frac{-1}{k(y^{\frac{k}{2}}-y^{-\frac{k}{2}})^2} \left(\frac{\frac{\beta\cdot E}{k}(-1)^{\frac{\beta^2}{k^2}+1}\Omega_{\frac{\beta}{k}}\left(y^{\frac{k}{2}}\right)}{\frac{\left(y^{\frac{k}{2}}\right)^{\frac{\beta\cdot E}{k}}-\left(y^{\frac{k}{2}}\right)^{-\frac{\beta\cdot E}{k}}}{y^{\frac{k}{2}}-y^{-\frac{k}{2}}}}-\sum_{g\geq 0}n_{g,\frac{\beta}{k}}(-1)^g(y^{\frac{k}{2}}-y^{-\frac{k}{2}})^{2g}\right)&\\
=\sum_{T\in \RT_\beta,\atop t^{-1}(0)\neq \emptyset} \Cont_{T}.
\end{split}
\end{equation}
Note that each
\[ \frac{-1}{(y^{\frac{k}{2}}-y^{-\frac{k}{2}})^2} \left(\frac{\frac{\beta\cdot E}{k}(-1)^{\frac{\beta^2}{k^2}+1}\Omega_{\frac{\beta}{k}}\left(y^{\frac{k}{2}}\right)}{\frac{\left(y^{\frac{k}{2}}\right)^{\frac{\beta\cdot E}{k}}-\left(y^{\frac{k}{2}}\right)^{-\frac{\beta\cdot E}{k}}}{y^{\frac{k}{2}}-y^{-\frac{k}{2}}}}-\sum_{g\geq 0}n_{g,\frac{\beta}{k}}(-1)^g(y^{\frac{k}{2}}-y^{-\frac{k}{2}})^{2g}\right)\]
can be derived from
\[\frac{-1}{(y^{\frac{1}{2}}-y^{-\frac{1}{2}})^2} \left(\frac{(\beta\cdot E)(-1)^{\beta^2+1 }\Omega_{\beta}\left(y^{\frac{1}{2}}\right)}{{\frac{\left(y^{\frac{1}{2}}\right)^{\beta\cdot E}-\left(y^{\frac{1}{2}}\right)^{-\beta\cdot E}}{y^{\frac{1}{2}}-y^{-\frac{1}{2}}}}}-\sum_{g\geq 0}n_{g,\beta}(-1)^g(y^{\frac{1}{2}}-y^{-\frac{1}{2}})^{2g}\right)\]
by replacing $y$ to $y^k$ and $\beta$ to $\frac{\beta}{k}$, and each
\begin{equation*}
\Cont_{T}\in\Q[y,y^{-1}]
\end{equation*}
by Lemma \ref{lem:key}. We can easily show by induction that
\begin{equation*}
\frac{-1}{(y^{\frac{1}{2}}-y^{-\frac{1}{2}})^2} \left(\frac{(\beta\cdot E)(-1)^{\beta^2+1}\Omega_{\beta}\left(y^{\frac{1}{2}}\right)}{{\frac{\left(y^{\frac{1}{2}}\right)^{\beta\cdot E}-\left(y^{\frac{1}{2}}\right)^{-\beta\cdot E}}{y^{\frac{1}{2}}-y^{-\frac{1}{2}}}}}-\sum_{g\geq 0}n_{g,\beta}(-1)^g(y^{\frac{1}{2}}-y^{-\frac{1}{2}})^{2g}\right)\in \Q[y,y^{-1}].
\end{equation*}
It then follows that 
\[\frac{\Omega_{\beta}\left(y^{\frac{1}{2}}\right)}{{\frac{\left(y^{\frac{1}{2}}\right)^{\beta\cdot E}-\left(y^{\frac{1}{2}}\right)^{-\beta\cdot E}}{y^{\frac{1}{2}}-y^{-\frac{1}{2}}}}}\in \Q[y,y^{-1}].\]
But $\Omega_{\beta}\left(y^{\frac{1}{2}}\right)\in \mathbb{Z}[y^{ \frac{1}{2}},y^{-\frac{1}{2}}]$, so actually
\[\frac{\Omega_{\beta}\left(y^{\frac{1}{2}}\right)}{{\frac{\left(y^{\frac{1}{2}}\right)^{\beta\cdot E}-\left(y^{\frac{1}{2}}\right)^{-\beta\cdot E}}{y^{\frac{1}{2}}-y^{-\frac{1}{2}}}}}\in \Z[y,y^{-1}].\]
Notice that both $\Omega_{\beta}$ and $$\frac{\left(y^{\frac{1}{2}}\right)^{\beta\cdot E}-\left(y^{\frac{1}{2}}\right)^{-\beta\cdot E}}{y^{\frac{1}{2}}-y^{-\frac{1}{2}}}$$
are palindromic. So 
\[\frac{\Omega_{\beta}\left(y^{\frac{1}{2}}\right)}{{\frac{\left(y^{\frac{1}{2}}\right)^{\beta\cdot E}-\left(y^{\frac{1}{2}}\right)^{-\beta\cdot E}}{y^{\frac{1}{2}}-y^{-\frac{1}{2}}}}}\]
is palindromic as well.
\end{proof}

\begin{rmk}\label{rmk:div+}
From the proof, we actually have a stronger conclusion:
\[\frac{-1}{(y^{\frac{1}{2}}-y^{-\frac{1}{2}})^2} \left(\frac{(\beta\cdot E)(-1)^{\beta^2+1}\Omega_{\beta}\left(y^{\frac{1}{2}}\right)}{{\frac{\left(y^{\frac{1}{2}}\right)^{\beta\cdot E}-\left(y^{\frac{1}{2}}\right)^{-\beta\cdot E}}{y^{\frac{1}{2}}-y^{-\frac{1}{2}}}}}-\sum_{g\geq 0}n_{g,\beta}(-1)^g(y^{\frac{1}{2}}-y^{-\frac{1}{2}})^{2g}\right)\in \Z[y,y^{-1}].\]
Note that here $n_{g,\beta} = 0$ if  $g>g(\beta)\coloneqq\frac{1}{2} \beta (K_S+\beta)+1$.
\end{rmk}

\subsection{Further analysis of the graph sum formula}
Assuming the refined sheaves/ Gromov-Witten correspondence (Conjecture \ref{conj:rsh/GW}), equation \eqref{eqn:sheaf/GW} in the proof of Theorem \ref{thm:main} provides a new way to computing the Betti numbers of $M_{\beta}$ via Gromov-Witten invariants of $K_S$ and a local elliptic curve. 
We can reformulate the RHS of \eqref{eqn:sheaf/GW} as 
\begin{equation}\label{eqn:contfm}
\sum_{T\in \RT_\beta,\atop t^{-1}(0)\neq \emptyset} \Cont_{T}=\sum_{k=0}^{\infty}\sum_{d_E,\beta_1,\cdots,\beta_k>0 \atop d_E [E]+\sum_{i=1}^k \beta_i=\beta}\frac{1}{|\Aut\boldsymbol{\beta}|}G^E_{d_E,\boldsymbol{\beta}}\prod_{i=1}^k \W^{K_S}_{\beta_i} ,
\end{equation}
where each $G^E_{d_E,\boldsymbol{\beta}}$ is a sum:
\[\sum_T\frac{|\Aut\boldsymbol{\beta}|}{|\Aut(T)|}
\prod_{e\in E_T}\frac{2\sin\frac{(d_2(e)\cdot E)\hbar}{2}}{d_2(e)\cdot E}\prod_{v_i\in t^{-1}(0)}\left(-F^{E,\tw}_{d^0(v_i),\{d_2(e_1^i),d_2(e_2^i),\cdots,d_2(e_{n_i}^i)\}}\right)\]
which takes over all the rooted trees $T$ such that $t^{-1}(0)\neq \emptyset$, the set of
labels 
$$\{d^1(v_i)|v_i\in \V_T\}=\boldsymbol{\beta}=\{\beta_1,\cdots,\beta_k\},\quad \text{ and } \sum_{v_i\in V_T} d^0(v_i)=d_E.$$
By Corollary \ref{cor:rat}, each $G^E_{d_E,\boldsymbol{\beta}}$ is a palindromic Laurent polynomial of $y$ under the change of variables $y=e^{\sqrt{-1} \hbar}$. We have the following $y$-degree bounds for $G^E_{d_E,\boldsymbol{\beta}}$:
\begin{lem}\label{lem:degbd}
For $d_E>1$, we have
\begin{align*}
-\frac{(d_E-1)(d_E-2)K_S^2}{2}-(d_E-1)&\sum_{i=1}^k(-K_S\cdot \beta_i)\leq \deg_y(G^E_{d_E,\boldsymbol{\beta}})\\
&\leq \frac{(d_E-1)(d_E-2)K_S^2}{2}+(d_E-1)\sum_{i=1}^k(-K_S\cdot \beta_i).
\end{align*}
\end{lem}

\begin{proof}
Observe that any rooted tree can be constructed by starting with a star-shaped tree and attaching rooted trees to its leaves. Using this recursive structure alongside Theorem \ref{thm:keythm}, we can derive the following equation for 
$G^E_{d_E,\boldsymbol{\beta}}$:
\begin{align}
& & 1  +\sum_{\rho} (-1)^{K_S^2 |\rho|} q^{|\rho|\cdot[E]} y^{K_S^2 c_{\rho}}\exp\Bigg(  \sum_{s\geq 0}\sum_{d_E\geq 0\atop \beta_1,\cdots,\beta_s>0}G_{d_E,\boldsymbol{\beta}}^E \cdot e_\rho(d_E,\boldsymbol{\beta})  q^{d_E[E]+\sum_{i=1}^s\beta_i}\frac{\prod_{i=1}^s p_{\beta_i}}{s!}\Bigg)\nonumber\\
& & = \exp\Bigg(-\sum_{s\geq 0}\sum_{d_E> 0\atop \beta_1,\cdots,\beta_s>0}G_{d_E,\boldsymbol{\beta}}^E q^{d_E[E]+\sum_{i=1}^s\beta_i}\frac{\prod_{i=1}^s p_{\beta_i}}{s!}\Bigg).
\label{eqn:treeid} \end{align}
Here $q$ and $p_{\beta}$ are formal variables, with $p_{\beta}$ indexed by nonzero effective curve classes $\beta$ of $S$. The exponents of $q$ are taken from the monoid of effective curve classes of $S$. The sum takes over all the partitions $\rho$. $e_\rho(d_E,\boldsymbol{\beta})$ is defined as
$$
e_\rho(d_E,\boldsymbol{\beta}) \coloneqq \tilde{e} (\rho,z )\cdot (e^{\frac{z}{2}}-e^{-\frac{z}{2}}) |_{z = (d_E K_S^2-\sum_{i=1}^s K_S\cdot \beta_i)\sqrt{-1}\hbar},
$$
where $\tilde{e} (\rho,z )$ is given by \eqref{eqn:tildeE}. We treat $e_\rho(d_E,\boldsymbol{\beta})$ as a Laurent polynomial of $y$, using the substitution $y=e^{i\hbar}$. Additionally, we define $G_{0,\boldsymbol{\beta}}^E$ as follows:
\[G_{0,\boldsymbol{\beta}}^E=
\begin{cases}
1, \text{ if }s=1,\\
0, \text{ otherwise,}
\end{cases}
\]
where $s$ is the cardinality of the set $\boldsymbol{\beta}=\{\beta_1,\cdots,\beta_s\}$.

One can easily deduce from \eqref{eqn:treeid} that
\begin{equation}\label{eqn:dE1fm}
G_{1,\boldsymbol{\beta}}^E=-(-1)^{K_S^2}\prod_{i=1}^s\Big(y^{-\frac{K_S\cdot\beta_i}{2}}-y^{\frac{K_S\cdot\beta_i}{2}}\Big)^2
\end{equation}
and 
\[G_{2,\boldsymbol{\beta}}^E=-3\cdot 2^{s-1}\prod_{i=1}^s\Big(y^{-\frac{K_S\cdot\beta_i}{2}}-y^{\frac{K_S\cdot\beta_i}{2}}\Big)^2.\]
The above computation shows that Lemma \ref{lem:degbd} holds for $d_E=2$ and $\forall k$.
In order to prove the lemma, we will perform induction on integers $d_E$. Assume that the lemma holds for $ 2\leq d_E<d$ and $\forall k$. We want to prove that 
\begin{align*}
-\frac{(d_E-1)(d_E-2)K_S^2}{2}-(d_E-1)&\sum_{i=1}^k(-K_S\cdot \beta_i)\leq \deg_y(G^E_{d_E,\boldsymbol{\beta}})\\
&\leq \frac{(d_E-1)(d_E-2)K_S^2}{2}+(d_E-1)\sum_{i=1}^k(-K_S\cdot \beta_i).
\end{align*}
holds for $d_E=d$ and $\forall k$. Since $G^E_{d,\boldsymbol{\beta}}$ is invariant under the change of $y\rightarrow y^{-1}$, we only need to prove that 
\begin{equation}\label{eqn:iduceq}
\deg_y(G^E_{d,\boldsymbol{\beta}})\leq \frac{(d-1)(d-2)K_S^2}{2}+(d-1)\sum_{i=1}^k(-K_S\cdot \beta_i)
\end{equation}
holds for $\forall k$. Let 
\[H(q,p_{\beta_i},y)\coloneqq\exp\Bigg(-\sum_{s\geq 0}\sum_{d_E> 0\atop \beta_1,\cdots,\beta_s>0}G_{d_E,\boldsymbol{\beta}}^E q^{d_E[E]+\sum_{i=1}^s\beta_i}\frac{\prod_{i=1}^s p_{\beta_i}}{s!}\Bigg)\]
and 
\[L(q,p_{\beta_i})\coloneqq \exp\Bigg(\sum_{\beta_i>0} q^{\beta_i}p_{\beta_i}\Bigg),\qquad I(q,p_{\beta_i},y)\coloneqq \frac{H}{L}.\]
Then using the equality:
\[\tilde{e} (\rho,z )\cdot (e^{\frac{z}{2}}-e^{-\frac{z}{2}})=\sum_{\square\in\rho}\left(e^{(c(\square)+1)z}-2e^{c(\square)z}+e^{(c(\square)-1)z}\right),\]
we could rewrite \eqref{eqn:treeid} as
\begin{equation}\label{eqn:treeid2}
\begin{split}
1+\sum_{\rho}\prod_{\square\in\rho}(-1)^{K_S^2}q^{[E]}y^{c(\square)K_S^2}\frac{I(qy^{-c(\square)K_S},p_{\beta_i},y)^2}{I(qy^{-(c(\square)+1)K_S},p_{\beta_i},y)I(qy^{-(c(\square)-1)K_S},p_{\beta_i},y)}&\\
=H(q,p_{\beta_i},y).
\end{split}
\end{equation}
Here each partition $\rho$ corresponds to a Young diagram. For a box $\square$ located at position $(i,j)$ within a Young diagram, the content is defined as $c(\square)=j-i$. When substituting the variable
$q\rightarrow qy^{nK_S}$ with $n\in \Z$, it is important to interpret this substitution carefully: the exponent of $q$ is always an effective curve class $\beta$. After performing the substitution $q\rightarrow qy^{nK_S}$, the term $q^{\beta}$ transforms into $q^{\beta}y^{n(K_S\cdot\beta)}$.

By induction, it is easy to see that inequality
\eqref{eqn:iduceq} is equivalent to 
\begin{equation}\label{eqn:iduceq2}
\deg_y\Big(\Big[q^{d[E]}\prod_{i=1}^k q^{\beta_i}p_{\beta_i}\Big]H(q,p_{\beta_i},y)\Big)\leq \frac{(d-1)(d-2)K_S^2}{2}+(d-1)\sum_{i=1}^k(-K_S\cdot \beta_i)
\end{equation}
where $\Big[q^{d[E]}\prod_{i=1}^k q^{\beta_i}p_{\beta_i}\Big]$ means we take the $q^{d[E]}\prod_{i=1}^k q^{\beta_i}p_{\beta_i}$-coefficient of $H$. The inequality \eqref{eqn:iduceq2} can be derived
by a detailed analysis of the LHS of \eqref{eqn:treeid2} as follows:

Given a partition $\rho=(\rho_1\geq \rho_2\geq \cdot \geq \rho_l)$, one could easily check by induction that if $l\geq 2$, i.e., $\rho $ contains at least 2 parts, then the $y$-degree of the $q^{d[E]}\prod_{i=1}^k q^{\beta_i}p_{\beta_i}$-coefficient in the summand
\[\prod_{\square\in\rho}(-1)^{K_S^2}q^{[E]}y^{c(\square)K_S^2}\frac{I(qy^{-c(\square)K_S},p_{\beta_i},y)^2}{I(qy^{-(c(\square)+1)K_S},p_{\beta_i},y)I(qy^{-(c(\square)-1)K_S},p_{\beta_i},y)}\]
on the LHS of \eqref{eqn:treeid2} is at most 
\[\frac{(d-1)(d-2)K_S^2}{2}-K_S^2+(d-1)\sum_{i=1}^k(-K_S\cdot \beta_i).\]
So we only need to show that the $y$-degree of the $q^{d[E]}\prod_{i=1}^k q^{\beta_i} p_{\beta_i}$-coefficient in the sum
\begin{equation}\label{eqn:1part}
\sum_{m=1}^d\prod_{\square\in (m)}(-1)^{K_S^2}q^{[E]}y^{c(\square)K_S^2}\frac{I(qy^{-c(\square)K_S},p_{\beta_i},y)^2}{I(qy^{-(c(\square)+1)K_S},p_{\beta_i},y)I(qy^{-(c(\square)-1)K_S},p_{\beta_i},y)}
\end{equation}
is at most 
\[\frac{(d-1)(d-2)K_S^2}{2}+(d-1)\sum_{i=1}^k(-K_S\cdot \beta_i)\]
where $(m)$ stands for the partition contains only one part $m$. We use $H_d^1(q,p_{\beta_i},y)$ to denote \eqref{eqn:1part} plus one. Then we can rewrite \eqref{eqn:1part} as 
\[(-1)^{K_S^2}q^{[E]}\frac{I(q,p_{\beta_i},y)^2}{I(qy^{K_S},p_{\beta_i},y)I(qy^{-K_S},p_{\beta_i},y)} H_{d-1}^1(qy^{-K_S},p_{\beta_i},y).\]
We use $I_d$ (resp. $H_d$ and $L_d$) to denote a truncation of $I$ (resp. $H$ and $L$) which contains terms in $I$ (resp. $H$ and $L$) whose $q$-degrees are at most $d[E]$ after a replacement of $p_{\beta_i}$ by $q^{-\beta_i}p_{\beta_i}$\footnote{After the replacement of $p_{\beta_i}$ by $q^{-\beta_i}p_{\beta_i}$, the power of $z$ is always in the form of $m[E]$, $m\in\Z_{\geq 0}$.}. So we only need to show that  the $y$-degree of the $q^{(d-1)[E]}\prod_{i=1}^k q^{\beta_i} p_{\beta_i}$-coefficient in
\[
\frac{I_{d-1}(q,p_{\beta_i},y)^2L_{d-1}(qy^{-K_S},p_{\beta_i},y)}{I_{d-1}(qy^{K_S},p_{\beta_i},y)}\frac{H_{d-1}^1(qy^{-K_S},p_{\beta_i},y)}{H_{d-1}(qy^{-K_S},p_{\beta_i},y)} 
\]
is at most 
\[\frac{(d-1)(d-2)K_S^2}{2}+(d-1)\sum_{i=1}^k(-K_S\cdot \beta_i).\]
This follows from the induction and the fact that $H_{d-1}(qy^{-K_S},p_{\beta_i},y)$ can be written as
$$H_{d-1}^1(qy^{-K_S},p_{\beta_i},y)+H_{d-1}^2(qy^{-K_S},p_{\beta_i},y)$$
where the $y$-degree of the $q^{m[E]}\prod_{i=1}^k q^{\beta_i} p_{\beta_i}$-coefficient in $H_{d-1}^2(qy^{-K_S},p_{\beta_i},y)$ is at most 
\[\frac{m(m-1)K_S^2}{2}+m\sum_{i=1}^k(-K_S\cdot \beta_i),\qquad m\leq d-1.\]
The latter $y$-degree bound follows from the fact that $H_{d-1}^2(q,p_{\beta_i},y)$ is a sum over partitions $\rho$ contains at least 2 parts and $|\rho|\leq d-1$, i.e.,
\[\sum_{\rho=(\rho_1\geq \cdots \geq\rho_l)\atop l\geq 2, |\rho|\leq d-1}\prod_{\square\in\rho}(-1)^{K_S^2}q^{[E]}y^{c(\square)K_S^2}\frac{I(qy^{-c(\square)K_S},p_{\beta_i},y)^2}{I(qy^{-(c(\square)+1)K_S},p_{\beta_i},y)I(qy^{-(c(\square)-1)K_S},p_{\beta_i},y)}.\]
\end{proof}

\section{A closed formula for leading Betti numbers}\label{sec:Leading}

In this section, we aim to give a closed formula for the leading Betti numbers when $S=\PP^2$ since the refined sheaves/Gromov-Witten correspondence is known to hold in this case. 
Specifically, we use $\Omega^{\mathbb P^2}_d$ to denote the shifted Poincar\'e polynomial of the moduli space of Gieseker stable sheaves on $\PP^2$ supported on curves with degree $d\in \Z_{>0}$ and Euler characteristic $1$.

\subsection{Degree analysis}\label{subsec:degbd}
We set
\[\mathcal  F_d(y)\coloneqq\sum_{g\geq 0}n_{g,d}(-1)^g(y^{\frac{1}{2}}-y^{-\frac{1}{2}})^{2g}\]
where $n_{g,d}$ are the genus $g$, degree $d$ Gopakumar-Vafa invariants of local $\PP^2$. Note that for fixed $d$,
\[n_{g,d}=0,\,\text{ if }g>g(d)\coloneqq\frac{(d-1)(d-2)}{2}. \]
Using Lemma \ref{lem:degbd}, we could deduce the following degree bound:
\begin{prop}\label{prop:degbd1}
Let 
$$
\mathcal X_d\coloneqq \frac{3d (-1)^{d^2+1}\Omega_{d}\left(y^{\frac{1}{2}}\right)}{{\frac{y^{\frac{3d}{2}} -y^{-\frac{3d}{2}}}{y^{\frac{1}{2}}-y^{-\frac{1}{2}}}}}- \mathcal  F_d(y) - \big(y^{\frac{3(d-3) }{2}}-y^{-\frac{3(d-3) }{2}}\big)^2 \cdot  \mathcal F_{d-3}(y).$$
Then for $d\geq 5$,
$$
\mathcal X_d \in  \  \Z[y,y^{-1}]_{ -g(d) +d-4 \leq \deg \leq g(d)-d+4 }.
$$
\end{prop}
\begin{proof}
By Theorem \ref{thm:main}, it is obvious that $\mathcal X_d\in  \Z[y,y^{-1}]$. Furthermore, $\mathcal X_d$ is invariant under the change of $y\rightarrow y^{-1}$. So we only need to show the upper bound, i.e.,
\begin{equation}\label{eqn:mainineq1}
\deg_y(\mathcal X_d)\leq g(d)-d+4,\quad \text{ for } d\geq 5.
\end{equation}

Combining equations \eqref{eqn:sheaf/GW} and \eqref{eqn:contfm}, we have
\begin{equation}\label{eqn:key2bd}
\begin{aligned}
\ & \sum_{k|d} \frac{(y^{\frac{1}{2}}-y^{-\frac{1}{2}})^2}{k(y^{\frac{k}{2}}-y^{-\frac{k}{2}})^2} \left(\frac{\frac{3d}{k}(-1)^{\frac{d^2}{k^2}+1}\Omega_{\frac{d}{k}}\left(y^{\frac{k}{2}}\right)}{\frac{\left(y^{\frac{k}{2}}\right)^{\frac{3d}{k}}-\left(y^{\frac{k}{2}}\right)^{-\frac{3d}{k}}}{y^{\frac{k}{2}}-y^{-\frac{k}{2}}}}-\sum_{g\geq 0}n_{g,\frac{d}{k}}(-1)^g(y^{\frac{k}{2}}-y^{-\frac{k}{2}})^{2g}\right)\\
= \ & \sum_{k=0}^{\infty}\sum_{d_E,d_1,\cdots,d_k>0 \atop 3d_E +\sum_{i=1}^k d_i=d}\frac{-(y^{\frac{1}{2}}-y^{-\frac{1}{2}})^2}{|\Aut\{d_1,\cdots,d_k\}|}G^E_{d_E,\{d_1,\cdots,d_k\}}\prod_{i=1}^k \W^{K_{\PP^2}}_{d_i}.
\end{aligned}
\end{equation}
after further multiplying both sides by $-(y^{\frac{1}{2}}-y^{-\frac{1}{2}})^2$. Note that if $k=1$, then the summand on the LHS of \eqref{eqn:key2bd} becomes
\[\frac{3d (-1)^{d^2+1}\Omega_{d}\left(y^{\frac{1}{2}}\right)}{{\frac{y^{\frac{3d}{2}} -y^{-\frac{3d}{2}}}{y^{\frac{1}{2}}-y^{-\frac{1}{2}}}}}- \mathcal  F_d(y).\]
If $k> 1$, then by Remark \ref{rmk:div+}, the summand
\[ \frac{(y^{\frac{1}{2}}-y^{-\frac{1}{2}})^2}{k(y^{\frac{k}{2}}-y^{-\frac{k}{2}})^2} \left(\frac{\frac{3d}{k}(-1)^{\frac{d^2}{k^2}+1}\Omega_{\frac{d}{k}}\left(y^{\frac{k}{2}}\right)}{\frac{\left(y^{\frac{k}{2}}\right)^{\frac{3d}{k}}-\left(y^{\frac{k}{2}}\right)^{-\frac{3d}{k}}}{y^{\frac{k}{2}}-y^{-\frac{k}{2}}}}-\sum_{g\geq 0}n_{g,\frac{d}{k}}(-1)^g(y^{\frac{k}{2}}-y^{-\frac{k}{2}})^{2g}\right)\in\Z[y,y^{-1}].\]
The $y$-degree upper bound of the above summand can be easily computed:
\[\frac{d^2}{2k}-\frac{3d}{2}+1.\]
One can easily verify that if $k\geq 2$, then
\begin{equation}\label{eqn:k2ineq}
\frac{d^2}{2k}-\frac{3d}{2}+1\leq g(d)-d+1<g(d)-d+4,\quad \forall d\in \Z_{>0}.
\end{equation}
For each summand on the RHS of \eqref{eqn:key2bd}, 
\[\frac{-(y^{\frac{1}{2}}-y^{-\frac{1}{2}})^2}{|\Aut\{d_1,\cdots,d_k\}|}G^E_{d_E,\{d_1,\cdots,d_k\}}\prod_{i=1}^k \W^{K_{\PP^2}}_{d_i}\in\Q[y,y^{-1}]\]
by Lemma \ref{lem:key}. If $d_E>1$, then by Lemma \ref{lem:degbd}, we have
\begin{align*}
     & \deg_y\Big((y^{\frac{1}{2}}-y^{-\frac{1}{2}})^2G^E_{d_E,\{d_1,\cdots,d_k\}}\prod_{i=1}^k \W^{K_{\PP^2}}_{d_i}\Big)\\
 \leq \ & 1+\frac{9(d_E-1)(d_E-2)}{2}+(d_E-1)\sum_{i=1}^k (3d_i)+\sum_{i=1}^k\Big(\frac{d_i^2}{2}-\frac{3d_i}{2}\Big)\\
 =   \  & \frac{(3d_E+\sum_{i=1}^k d_i)^2}{2}-\frac{9(3d_E+\sum_{i=1}^k d_i)}{2}+10-\sum_{1\leq i<j \leq k}d_id_j\\
 =  \  & g(d)-(3d-9+\sum_{1\leq i<j \leq k}d_id_j) \quad \text{ Here we use } 3d_E+\sum_{i=1}^k d_i=d\\
 <  \  & g(d)-d+4,\quad \text{if }d\geq 5.
\end{align*}
If $d_E=1$, then by \eqref{eqn:dE1fm}, one can easily compute that 
\begin{equation}\label{eqn:dE1deg}
\deg_y\Big((y^{\frac{1}{2}}-y^{-\frac{1}{2}})^2G^E_{1,\{d_1,\cdots,d_k\}}\prod_{i=1}^k \W^{K_{\PP^2}}_{d_i}\Big)=g(d)-\sum_{1\leq i<j \leq k}d_i d_j.
\end{equation}
Note that in this case, $\sum_{i=1}^k d_i=d-3$ and $d_i\in\Z_{>0}$, one can easily show that 
\[\sum_{1\leq i<j \leq k} d_i d_j\geq d-4,\,\text{ if }k\geq 2.\]
So the only summand on the RHS of \eqref{eqn:key2bd} whose $y$-degree is greater than $g(d)-d+4$ is the case $d_E=1$ and $k=1$, i.e.,
\begin{align*}
     & -(y^{\frac{1}{2}}-y^{-\frac{1}{2}})^2 G_{1,{d-3}}^E\W^{K_{\PP^2}}_{d-3}\\
     \stackrel{\text{by \eqref{eqn:dE1fm}}}{=} \  & \big(y^{\frac{3(d-3) }{2}}-y^{-\frac{3(d-3) }{2}}\big)^2\sum_{k|(d-3)}\frac{(y^{\frac{1}{2}}-y^{-\frac{1}{2}})^2}{k(y^{\frac{k}{2}}-y^{-\frac{k}{2}})^2} \mathcal  F_{\frac{d-3}{k}}(y^{k}).
\end{align*}
Using a similar inequality as \eqref{eqn:k2ineq}, we have
\[\deg_y\Big(\big(y^{\frac{3(d-3) }{2}}-y^{-\frac{3(d-3) }{2}}\big)^2\frac{(y^{\frac{1}{2}}-y^{-\frac{1}{2}})^2}{k(y^{\frac{k}{2}}-y^{-\frac{k}{2}})^2} \mathcal  F_{\frac{d-3}{k}}(y^{k})\Big)\leq g(d)-d+4,\quad \text{ if }k>1.\]
From the above analysis, we may conclude that only the term
\[\big(y^{\frac{3(d-3) }{2}}-y^{-\frac{3(d-3) }{2}}\big)^2 \cdot  \mathcal F_{d-3}(y)\]
on the LHS of \eqref{eqn:key2bd} has degree greater than $g(d)-d+4$. Now inequality \eqref{eqn:mainineq1} follows from the above analysis.
\end{proof}

\begin{rmk}
For $d<5$, we have a better degree bound. Actually, by the computation in Section \ref{sec:Nume}, we have
\[\mathcal X_d=
\begin{cases}
0,\quad \text{ if }d=1,2,4,\\
-(y^{\frac{1}{2}}-y^{-\frac{1}{2}})^2,\quad \text{ if }d=3.
\end{cases}
\]
where we set $\mathcal F_d=0$ if $d\leq 0$.
\end{rmk}

Furthermore, we have the following stronger result when $d\geq 6$:
\begin{prop} \label{prop:degbd2}
Let 
$$
\mathcal Y_d\coloneqq\mathcal X_d  + n_{0,1} \left( \frac{ y^{\frac{3 }{2}}-y^{-\frac{3 }{2}}}{y^{\frac{1}{2}}-y^{-\frac{1 }{2}}}\right)^2 \big(y^{\frac{3(d-4) }{2}}-y^{-\frac{3(d-4) }{2}}\big)^2\cdot  \mathcal F_{d-4}
$$
where $n_{0,1}=3$ is the genus $0$, degree $1$ Gopakumar-Vafa invariants of local $\PP^2$. Then for $d\geq 6$,
$$
\mathcal Y_d \in \Z[y,y^{-1}]_{ -g(d) +2d-10 \leq \deg \leq g(d)-2d+10 }.
$$
\end{prop}
\begin{proof}
The proof is similar to Proposition \ref{prop:degbd1}. The only difference is that there will be a new summand on the RHS of \eqref{eqn:key2bd} whose $y$-degree is greater than $g(d)-2d+10$:
\begin{align*}
& \frac{-(y^{\frac{1}{2}}-y^{-\frac{1}{2}})^2}{|\Aut\{1,d-4\}|}G^E_{1,\{1,d-4\}}\W^{K_{\PP^2}}_{1}\W^{K_{\PP^2}}_{d-4}\\
 \stackrel{\text{by \eqref{eqn:dE1fm}}}{=}  \ & - n_{0,1} \left( \frac{ y^{\frac{3 }{2}}-y^{-\frac{3 }{2}}}{y^{\frac{1}{2}}-y^{-\frac{1 }{2}}}\right)^2 \big(y^{\frac{3(d-4) }{2}}-y^{-\frac{3(d-4) }{2}}\big)^2\sum_{k|(d-4)}\frac{(y^{\frac{1}{2}}-y^{-\frac{1}{2}})^2}{k(y^{\frac{k}{2}}-y^{-\frac{k}{2}})^2} \mathcal  F_{\frac{d-4}{k}}(y^{k}).
\end{align*}
One can similarly show that only the $k=1$ term
\[-n_{0,1} \left( \frac{ y^{\frac{3 }{2}}-y^{-\frac{3 }{2}}}{y^{\frac{1}{2}}-y^{-\frac{1 }{2}}}\right)^2 \big(y^{\frac{3(d-4) }{2}}-y^{-\frac{3(d-4) }{2}}\big)^2\cdot  \mathcal F_{d-4}\]
has $y$-degree greater than $g(d)-2d+10$. It then follows that
\[\deg_y\mathcal Y_d\leq g(d)-2d+10.\]
Since $\mathcal Y_d$ is also invariant under the change of  $y\rightarrow y^{-1}$, the lower bound also follows.
\end{proof}

\subsection{Proof of Theorem \ref{thm:leadingBetti}}
To prove Theorem \ref{thm:leadingBetti}, it is crucial to understand the leading Gopakumar-Vafa invariants of $K_{\PP^2}$. These invariants were explicitly computed by the first author and Zhou using the topological vertex formula in \cite{GZ15}:
\begin{thm}[\cite{GZ15}] \label{thm:gvbd}
Let
\begin{align*}
\mathcal Z_d  & =  \frac{(1-y)^2 }{ \prod_{k>0}(1-y^{k})^3}  \cdot \left(  \binom{d+2}{2}  -3 \sum_{i\geq 1 } \frac{y^i}{(1-y^i)^2} \right),\\
\mathcal Z'_d & = \frac{ (1-y^3)}{ \prod_{k>0}(1-y^{k})^3} \cdot \left(  \binom{d+1}{2}  -3 \sum_{i\geq 1 } \frac{y^i}{(1-y^i)^2} -3  \frac{y^3}{(1-y^3) }\right),
\end{align*}
then  we have 
\begin{align*}
 (-1)^{d-1}  y^{g(d)}\cdot \mathcal F_d - \mathcal Z_d  \ & \in \ y^{ d-1}\cdot \mathbb \Z[[y]],\quad \text{if } d\geq 1,\\
(-1)^{d-1} y^{g(d)}\cdot \mathcal F_d - \mathcal Z_d +3 \,y^{d-1}\cdot \mathcal Z'_d \ & \in \ y^{2d-4}\cdot \mathbb \Z[[y]],\quad \text{if } d\geq 3.
\end{align*}
\end{thm}

These $\mathcal Z_d$ and $\mathcal Z'_d$ satisfy the following equations:
\begin{lem}\label{lem:1st}
For $d\geq 4$,
\begin{align*}
   \frac{1}{3d} \left( \mathcal Z_d    -  \mathcal Z_{d-3}  \right)   =    \prod_{k>0}\frac{1}{(1-y^{k})(1-y^{k+1})^2}.
\end{align*}
\end{lem}
\begin{proof}
This directly follows from Theorem \ref{thm:gvbd}.
\end{proof}

\begin{lem}\label{lem:2nd}
For $d\geq 5$,
 \begin{align*}
 & \ \frac{1}{3d} \cdot \left(  - y^{d-1} \cdot \mathcal Z'_d   +  y^{d-4} \cdot \mathcal Z'_{d-3}    -y^{d-4}(1+y+y^2)^2 \cdot \mathcal Z_{d-4}    \right) \\
 = & \  \frac{-y^{d-1}(1+y+y^2)}{\prod_{k>0}(1-y^{k})^2(1-y^{k+1})}.
\end{align*}
\end{lem}
\begin{proof}
This just follows from the direct computation:
 \begin{align*}
 \mathrm{LHS}  \ =  \ \frac{y^{d-1}}{3d} \bigg[  \ &  \frac{ -(1-y^3)}{ \prod_{k>0}(1-y^{k})^3} \cdot \left(  \binom{d+1}{2}  -3 \sum_{i\geq 1 } \frac{y^i}{(1-y^i)^2} -3  \frac{y^3}{(1-y^3) }\right)  \\
 \ &  +\frac{ y^{-3} \,(1-y^3)}{ \prod_{k>0}(1-y^{k})^3} \cdot \left(  \binom{d-2}{2}  -3 \sum_{i\geq 1 } \frac{y^i}{(1-y^i)^2} -3  \frac{y^3}{(1-y^3) }\right)  \\
 \ &   - \frac{ y^{-3}(1-y^3)^2}{ \prod_{k>0}(1-y^{k})^3} \cdot \left(  \binom{d-2}{2}  -3 \sum_{i\geq 1 } \frac{y^i}{(1-y^i)^2} \right)  \bigg]\\
\  =  \ \frac{y^{d-1}}{3d} \bigg[  \ &  \frac{ -(1-y^3)}{ \prod_{k>0}(1-y^{k})^3} \cdot \left(  \binom{d+1}{2}  -3 \sum_{i\geq 1 } \frac{y^i}{(1-y^i)^2} -3  \frac{y^3}{(1-y^3) }\right)  \\
 \ &  +\frac{ y^{-3} \,(1-y^3)}{ \prod_{k>0}(1-y^{k})^3} \cdot \left(  \binom{d-2}{2}  -3 \sum_{i\geq 1 } \frac{y^i}{(1-y^i)^2} -3  \frac{y^3}{(1-y^3) }\right)  \\
 \ &   -\frac{ y^{-3}(1-y^3)}{ \prod_{k>0}(1-y^{k})^3} \cdot \left(  \binom{d-2}{2}  -3 \sum_{i\geq 1 } \frac{y^i}{(1-y^i)^2} \right)\\  
 \ & +  \frac{(1-y^3)}{ \prod_{k>0}(1-y^{k})^3} \cdot \left(  \binom{d-2}{2}  -3 \sum_{i\geq 1 } \frac{y^i}{(1-y^i)^2} \right)
 \bigg]\\
 \  =  \frac{y^{d-1}}{3d} \cdot\bigg[  \ &  \frac{- (1-y^3)}{ \prod_{k>0}(1-y^{k})^3} \cdot \left(  \binom{d+1}{2}  - \binom{d-2}{2} -3  \frac{y^3}{(1-y^3) }+ 3 \frac{1}{(1-y^3) }\right) \bigg]   \\
\  =  \frac{y^{d-1}}{3d} \cdot\bigg[  \ &  \frac{- (1-y^3)}{ \prod_{k>0}(1-y^{k})^3} \cdot \left(  3d-3 - 3\frac{y^3}{(1-y^3) }+ 3 \frac{1}{(1-y^3) }\right) \bigg]  \ =   \mathrm{RHS}.
\end{align*}
\end{proof}

We are ready to give a closed formula for the leading Betti numbers of $M_d$.
\begin{thm}[=Theorem \ref{thm:leadingBetti}]\label{thm:2nd}
For $d\geq 6$, we have
$$
\hat{\Omega}^{\mathbb P^2}_d =  \prod_{k>0} \frac{1}{(1-y^{k})(1-y^{k+1})^{2}} \Big( 1 - 3\,y^{d-1}\cdot \frac{ 1+  y + y^2   }{1-y}  \Big) \bigg|_{y^{\leq 2d-11}} +O ({y^{ 2d-10}}).
$$
\end{thm}

\begin{proof}
Recall that 
$$\hat{\Omega}^{\mathbb P^2}_d= y^{\frac{(d-1)(d-2)}{2}}\cdot \frac{\Omega^{\mathbb P^2}_d}{{{\frac{\left(y^{\frac{1}{2}}\right)^{3d}-\left(y^{\frac{1}{2}}\right)^{-3d}}{y^{\frac{1}{2}}-y^{-\frac{1}{2}}}}}}.$$
By Proposition \ref{prop:degbd2}, to compute the coefficients of $\hat{\Omega}^{\mathbb P^2}_d$ up to degree $2d-11$, we only need to compute the coefficients of 
\[(-1)^{d-1}y^{g(d)}\mathcal F_d+(-1)^{d-1}y^{g(d)-3d+9}\mathcal F_{d-3}-(-1)^{d-1}n_{0,1}(1+y+y^2)^2 y^{g(d)-3d+10}\mathcal F_{d-4}\]
up to degree $2d-11$. By Theorem \ref{thm:gvbd}, it equals to
\begin{align*}
 & \mathcal Z_d-3y^{d-1}\mathcal Z'_{d}-(\mathcal Z_{d-3}-3y^{d-4}\mathcal Z'_{d-3})-3y^{d-4}(1+y+y^2)^2\mathcal Z_{d-4}+O ({y^{ 2d-10}})\\
= \ & \mathcal Z_d-\mathcal Z_{d-3}+3 \left(  - y^{d-1} \cdot \mathcal Z'_d   +  y^{d-4} \cdot \mathcal Z'_{d-3}    -y^{d-4}(1+y+y^2)^2 \cdot \mathcal Z_{d-4}    \right)+O ({y^{ 2d-10}}).
\end{align*}
After further dividing by $3d$ and using Lemmas \ref{lem:1st} and \ref{lem:2nd}, we can deduce that
$$
\hat{\Omega}^{\mathbb P^2}_d =  \prod_{k>0} \frac{1}{(1-y^{k})(1-y^{k+1})^{2}} \Big( 1 - 3\,y^{d-1}\cdot \frac{ 1+  y + y^2   }{1-y}  \Big) \bigg|_{y^{\leq 2d-11}} +O ({y^{ 2d-10}}).
$$
for $d\geq 6$.
\end{proof}

\section{Numerical results}\label{sec:Nume}
In this section, a detailed computation of the shifted Poincar\'e polynomials $\Omega_{d}^{\PP^2}$ of moduli space of one dimensional sheaves on $\PP^2$ is provided for $d\leq 6$ using our new method. We also give a list of $\Omega_{d}^{\PP^2}$ up to $d\leq 10$ at the end. The new method requires a prior knowledge of the
Gopakumar-Vafa invariants $n_{g,d}$ of local $\PP^2$ which comes from \cite{HKR}. 

For $d=1,2$, we actually have 
\[\frac{3d(-1)^{d^2+1}\Omega_{d}^{\mathbb{P}^2}\left(y^{\frac{1}{2}}\right)}{{\frac{\left(y^{\frac{1}{2}}\right)^{3d}-\left(y^{\frac{1}{2}}\right)^{-3d}}{y^{\frac{1}{2}}-y^{-\frac{1}{2}}}}}-\sum_{g\geq 0}n_{g,d}(-1)^g(y^{\frac{1}{2}}-y^{-\frac{1}{2}})^{2g}=0.\]
This implies that 
\begin{align*}
\Omega_{1}^{\mathbb{P}^2}&={\frac{y^{\frac{3}{2}}-y^{-\frac{3}{2}}}{y^{\frac{1}{2}}-y^{-\frac{1}{2}}}}\,,\\
\Omega_{2}^{\mathbb{P}^2}&={\frac{y^{3}-y^{-3}}{y^{\frac{1}{2}}-y^{-\frac{1}{2}}}}\,.\\
\end{align*}

For $d=3$, by \eqref{eqn:sheaf/GW} and induction, we have
\[\frac{-1}{(y^{\frac{1}{2}}-y^{-\frac{1}{2}})^2} \left(\frac{3d(-1)^{d^2+1}\Omega_{d}^{\mathbb{P}^2}\left(y^{\frac{1}{2}}\right)}{{\frac{\left(y^{\frac{1}{2}}\right)^{3d}-\left(y^{\frac{1}{2}}\right)^{-3d}}{y^{\frac{1}{2}}-y^{-\frac{1}{2}}}}}-\sum_{g\geq 0}n_{g,d}(-1)^g(y^{\frac{1}{2}}-y^{-\frac{1}{2}})^{2g}\right)=\sum_{T\in \RT_d,\atop t^{-1}(0)\neq \emptyset} \Cont_{T}\,.\]
There is only one rooted tree $T$ on the RHS:

\begin{center}
\begin{tikzpicture}[x=0.75pt,y=0.75pt,yscale=-1,xscale=1]

%Shape: Circle [id:dp5446169583573187] 
\draw   (100,159.5) .. controls (100,153.7) and (104.7,149) .. (110.5,149) .. controls (116.3,149) and (121,153.7) .. (121,159.5) .. controls (121,165.3) and (116.3,170) .. (110.5,170) .. controls (104.7,170) and (100,165.3) .. (100,159.5) -- cycle ;

% Text Node
\draw (105.5,152) node [anchor=north west][inner sep=0.75pt]   [align=left] {1};

\end{tikzpicture}
\end{center}
The corresponding contribution is 
\[\Cont_T=-F^{E,\tw}_{1,\emptyset}=1.\]
Combining with the facts that $n_{0,3}=27$, $n_{1,3}=-10$, we get 
\[\Omega_{3}^{\mathbb{P}^2}=\frac{1}{y}(1+y+y^2){\frac{y^{\frac{9}{2}}-y^{-\frac{9}{2}}}{y^{\frac{1}{2}}-y^{-\frac{1}{2}}}}\,.\]

For $d=4$, still by \eqref{eqn:sheaf/GW} and induction, we have
\[\frac{-1}{(y^{\frac{1}{2}}-y^{-\frac{1}{2}})^2} \left(\frac{3d(-1)^{d^2+1}\Omega_{d}^{\mathbb{P}^2}\left(y^{\frac{1}{2}}\right)}{{\frac{\left(y^{\frac{1}{2}}\right)^{3d}-\left(y^{\frac{1}{2}}\right)^{-3d}}{y^{\frac{1}{2}}-y^{-\frac{1}{2}}}}}-\sum_{g\geq 0}n_{g,d}(-1)^g(y^{\frac{1}{2}}-y^{-\frac{1}{2}})^{2g}\right)=\sum_{T\in \RT_d,\atop t^{-1}(0)\neq \emptyset} \Cont_{T}\, .\]
There is still only one rooted tree 
\begin{center}
\begin{tikzpicture}[x=0.75pt,y=0.75pt,yscale=-1,xscale=1]

%Shape: Circle [id:dp5446169583573187] 
\draw   (100,159.5) .. controls (100,153.7) and (104.7,149) .. (110.5,149) .. controls (116.3,149) and (121,153.7) .. (121,159.5) .. controls (121,165.3) and (116.3,170) .. (110.5,170) .. controls (104.7,170) and (100,165.3) .. (100,159.5) -- cycle ;

%Shape: Square [id:dp05663682075637111] 
\draw   (197.5,148.5) -- (218,148.5) -- (218,168.5) -- (197.5,168.5) -- cycle ;

%Straight Lines [id:da7070219457632385] 
\draw    (121,159.5) -- (197.5,159.5) ;
%Straight Lines [id:da891242303205514] 

% Text Node
\draw (105.5,152) node [anchor=north west][inner sep=0.75pt]   [align=left] {1};
% % Text Node
% \draw (204,64) node [anchor=north west][inner sep=0.75pt]   [align=left] {1};
% % Text Node
\draw (203.5,152) node [anchor=north west][inner sep=0.75pt]   [align=left] {1};
% % Text Node
% \draw (203,237) node [anchor=north west][inner sep=0.75pt]   [align=left] {6};
% % Text Node
% \draw (290,31) node [anchor=north west][inner sep=0.75pt]   [align=left] {2};
% % Text Node
% \draw (290,90) node [anchor=north west][inner sep=0.75pt]   [align=left] {2};
% % Text Node
% % \draw (289,151) node [anchor=north west][inner sep=0.75pt]   [align=left] {1};
% % Text Node
% \draw (240,34) node [anchor=north west][inner sep=0.75pt]   [align=left] {2};
% % Text Node
% \draw (239,93) node [anchor=north west][inner sep=0.75pt]   [align=left] {2};
% % Text Node
% \draw (146,91) node [anchor=north west][inner sep=0.75pt]   [align=left] {7};
% % Text Node
% % \draw (241,142) node [anchor=north west][inner sep=0.75pt]   [align=left] {1};
% % Text Node
\draw (155,145) node [anchor=north west][inner sep=0.75pt]   [align=left] {1};
% % Text Node
% \draw (147,219) node [anchor=north west][inner sep=0.75pt]   [align=left] {6};
\end{tikzpicture}
\end{center}

The corresponding contribution is 
\[-\frac{2\sin(\frac{3\hbar}{2})}{3}F^{E,\tw}_{1,\{1\}}\W_1=-\frac{y^{3/2}-y^{-3/2}}{3i}\cdot\frac{3(y^{3/2}-y^{-3/2})}{i}\cdot\frac{-3}{(y^{1/2}-y^{-1/2})^2}=-3\left(\frac{y^{3/2}-y^{-3/2}}{y^{1/2}-y^{-1/2}}\right)^2.\]
Combining with the fact that 
\[n_{0,4}=-192,\quad n_{1,4}=231,\quad n_{2,4}=-102,\quad n_{3,4}=15,\]
we may deduce that
\[\Omega_{4}^{\mathbb{P}^2}=\frac{1}{y^{3}}(1+y+4y^2+4y^3+4y^4+y^5+y^6){\frac{y^{6}-y^{-6}}{y^{\frac{1}{2}}-y^{-\frac{1}{2}}}}\,.\]

For $d=5$, still by \eqref{eqn:sheaf/GW} and induction, we have
\[\frac{-1}{(y^{\frac{1}{2}}-y^{-\frac{1}{2}})^2} \left(\frac{3d(-1)^{d^2+1}\Omega_{d}^{\mathbb{P}^2}\left(y^{\frac{1}{2}}\right)}{{\frac{\left(y^{\frac{1}{2}}\right)^{3d}-\left(y^{\frac{1}{2}}\right)^{-3d}}{y^{\frac{1}{2}}-y^{-\frac{1}{2}}}}}-\sum_{g\geq 0}n_{g,d}(-1)^g(y^{\frac{1}{2}}-y^{-\frac{1}{2}})^{2g}\right)=\sum_{T\in \RT_d,\atop t^{-1}(0)\neq \emptyset} \Cont_{T}
\,.\]
Now there are two rooted trees:
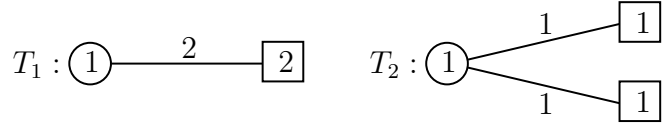
\begin{center}
\begin{tikzpicture}[x=0.75pt,y=0.75pt,yscale=-1,xscale=1]

%Shape: Circle [id:dp5446169583573187] 
\draw   (100,159.5) .. controls (100,153.7) and (104.7,149) .. (110.5,149) .. controls (116.3,149) and (121,153.7) .. (121,159.5) .. controls (121,165.3) and (116.3,170) .. (110.5,170) .. controls (104.7,170) and (100,165.3) .. (100,159.5) -- cycle ;

\draw   (280,159.5) .. controls (280,153.7) and (284.7,149) .. (290.5,149) .. controls (296.3,149) and (301,153.7) .. (301,159.5) .. controls (301,165.3) and (296.3,170) .. (290.5,170) .. controls (284.7,170) and (280,165.3) .. (280,159.5) -- cycle ;

%Shape: Square [id:dp05663682075637111] 
\draw   (197.5,148.5) -- (218,148.5) -- (218,168.5) -- (197.5,168.5) -- cycle ;

%Shape: Square [id:dp05663682075637111] 
\draw   (377.5,128.5) -- (398,128.5) -- (398,148.5) -- (377.5,148.5) -- cycle ;

%Shape: Square [id:dp05663682075637111] 
\draw   (377.5,168.5) -- (398,168.5) -- (398,188.5) -- (377.5,188.5) -- cycle ;

%Straight Lines [id:da7070219457632385] 
\draw    (121,159.5) -- (197.5,159.5) ;
%Straight Lines [id:da891242303205514] 

%Straight Lines [id:da7070219457632385] 
\draw    (301,161.5) -- (377.5,179.5) ;
%Straight Lines [id:da891242303205514] 

%Straight Lines [id:da7070219457632385] 
\draw    (301,157.5) -- (377.5,139.5) ;
%Straight Lines [id:da891242303205514] 

% Text Node
\draw (70,152) node [anchor=north west][inner sep=0.75pt]   [align=left] {$T_1:$};
% Text Node
\draw (250,152) node [anchor=north west][inner sep=0.75pt]   [align=left] {$T_2:$};
% Text Node
\draw (106,152) node [anchor=north west][inner sep=0.75pt]   [align=left] {1};
% Text Node
\draw (286,152) node [anchor=north west][inner sep=0.75pt]   [align=left] {1};
% % Text Node
\draw (204,152) node [anchor=north west][inner sep=0.75pt]   [align=left] {2};
% % Text Node
\draw (384,172) node [anchor=north west][inner sep=0.75pt]   [align=left] {1};
% % Text Node
\draw (384,132) node [anchor=north west][inner sep=0.75pt]   [align=left] {1};

% % Text Node
\draw (155,145) node [anchor=north west][inner sep=0.75pt]   [align=left] {2};
% % Text Node
\draw (335,133) node [anchor=north west][inner sep=0.75pt]   [align=left] {1};
% % Text Node
\draw (335,173) node [anchor=north west][inner sep=0.75pt]   [align=left] {1};
% % Text Node
% \draw (147,219) node [anchor=north west][inner sep=0.75pt]   [align=left] {6};
\end{tikzpicture}
\end{center}

The contribution of $T_1$ is 
\[\Cont_{T_1}=-\frac{2\sin(3\hbar)}{6}F^{E,\tw}_{1,\{2\}}\W_2=-\frac{3(y^3-y^{-3})^2}{2(y-y^{-1})^2}+\frac{6(y^3-y^{-3})^2}{(y^{1/2}-y^{-1/2})^2}\,.\]
The contribution of $T_2$ is 
\[\Cont_{T_2}=-\frac{1}{2}\left(\frac{2\sin(3\hbar/2)}{3}\right)^2 F^{E,\tw}_{1,\{1,1\}}(\W_1)^2=\frac{9(y^{3/2}-y^{-3/2})^4}{2(y^{1/2}-y^{-1/2})^4}\,.\]
Combining with the fact that 
\[n_{0,5}=1695,\, n_{1,5}=-4452,\, n_{2,5}=5430,\, n_{3,5}=-3672,\, n_{4,5}=1386,\, n_{5,5}=-270,\, n_{6,5}=21,\]
we may deduce that $\Omega_{5}^{\mathbb{P}^2}$ equals
\[\frac{1}{y^{6}}(1+y+4y^2+7y^3+13y^4+19y^5+23y^6+19y^7+13y^8+7y^9+4y^{10}+y^{11}+y^{12})\frac{y^{15/2}-y^{-15/2}}{y^{1/2}-y^{-1/2}}\,.\]

For $d=6$, by induction, the LHS of \eqref{eqn:sheaf/GW} becomes
\[\frac{1}{2}+\frac{-1}{(y^{\frac{1}{2}}-y^{-\frac{1}{2}})^2} \left(\frac{3d(-1)^{d^2+1}\Omega_{d}^{\mathbb{P}^2}\left(y^{\frac{1}{2}}\right)}{{\frac{\left(y^{\frac{1}{2}}\right)^{3d}-\left(y^{\frac{1}{2}}\right)^{-3d}}{y^{\frac{1}{2}}-y^{-\frac{1}{2}}}}}-\sum_{g\geq 0}n_{g,d}(-1)^g(y^{\frac{1}{2}}-y^{-\frac{1}{2}})^{2g}\right).\]
While the RHS of \eqref{eqn:sheaf/GW} contains 5 rooted trees:
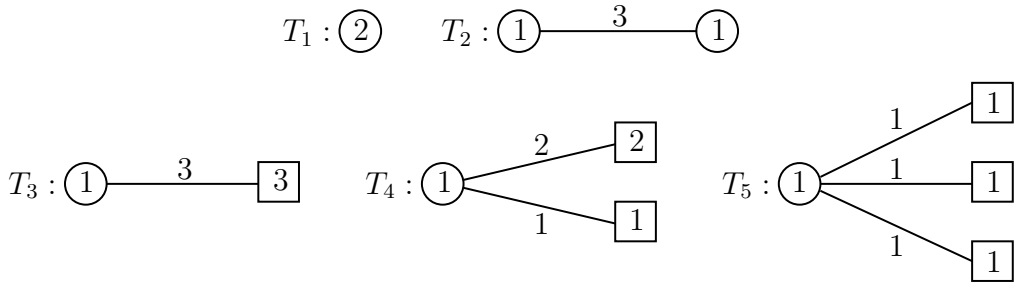
\begin{center}
\begin{tikzpicture}[x=0.75pt,y=0.75pt,yscale=-1,xscale=1]

%Shape: Circle [id:dp5446169583573187] 
\draw   (100,159.5) .. controls (100,153.7) and (104.7,149) .. (110.5,149) .. controls (116.3,149) and (121,153.7) .. (121,159.5) .. controls (121,165.3) and (116.3,170) .. (110.5,170) .. controls (104.7,170) and (100,165.3) .. (100,159.5) -- cycle ;

\draw   (180,159.5) .. controls (180,153.7) and (184.7,149) .. (190.5,149) .. controls (196.3,149) and (201,153.7) .. (201,159.5) .. controls (201,165.3) and (196.3,170) .. (190.5,170) .. controls (184.7,170) and (180,165.3) .. (180,159.5) -- cycle ;

\draw   (280,159.5) .. controls (280,153.7) and (284.7,149) .. (290.5,149) .. controls (296.3,149) and (301,153.7) .. (301,159.5) .. controls (301,165.3) and (296.3,170) .. (290.5,170) .. controls (284.7,170) and (280,165.3) .. (280,159.5) -- cycle ;

% %Shape: Square [id:dp05663682075637111] 
% \draw   (197.5,148.5) -- (218,148.5) -- (218,168.5) -- (197.5,168.5) -- cycle ;

% %Shape: Square [id:dp05663682075637111] 
% \draw   (377.5,128.5) -- (398,128.5) -- (398,148.5) -- (377.5,148.5) -- cycle ;

% %Shape: Square [id:dp05663682075637111] 
% \draw   (377.5,168.5) -- (398,168.5) -- (398,188.5) -- (377.5,188.5) -- cycle ;

%Straight Lines [id:da7070219457632385] 
\draw    (201,159.5) -- (280,159.5) ;
%Straight Lines [id:da891242303205514] 

% %Straight Lines [id:da7070219457632385] 
% \draw    (301,161.5) -- (377.5,179.5) ;
% %Straight Lines [id:da891242303205514] 

% %Straight Lines [id:da7070219457632385] 
% \draw    (301,157.5) -- (377.5,139.5) ;
%Straight Lines [id:da891242303205514] 

% Text Node
\draw (70,152) node [anchor=north west][inner sep=0.75pt]   [align=left] {$T_1:$};
% Text Node
\draw (150,152) node [anchor=north west][inner sep=0.75pt]   [align=left] {$T_2:$};
% Text Node
\draw (105.5,152) node [anchor=north west][inner sep=0.75pt]   [align=left] {2};
% Text Node
\draw (186,152) node [anchor=north west][inner sep=0.75pt]   [align=left] {1};
% % Text Node
\draw (286,152) node [anchor=north west][inner sep=0.75pt]   [align=left] {1};
% % Text Node
\draw (236,145) node [anchor=north west][inner sep=0.75pt]   [align=left] {3};
% % % Text Node
% \draw (384,132) node [anchor=north west][inner sep=0.75pt]   [align=left] {1};

% % % Text Node
% \draw (155,146) node [anchor=north west][inner sep=0.75pt]   [align=left] {2};
% % % Text Node
% \draw (335,135) node [anchor=north west][inner sep=0.75pt]   [align=left] {1};
% % % Text Node
% \draw (335,173) node [anchor=north west][inner sep=0.75pt]   [align=left] {1};
% % Text Node
% \draw (147,219) node [anchor=north west][inner sep=0.75pt]   [align=left] {6};
\end{tikzpicture}
\end{center}
\begin{center}
\begin{tikzpicture}[x=0.75pt,y=0.75pt,yscale=-1,xscale=1]

%Shape: Circle [id:dp5446169583573187] 
\draw   (100,159.5) .. controls (100,153.7) and (104.7,149) .. (110.5,149) .. controls (116.3,149) and (121,153.7) .. (121,159.5) .. controls (121,165.3) and (116.3,170) .. (110.5,170) .. controls (104.7,170) and (100,165.3) .. (100,159.5) -- cycle ;

\draw   (280,159.5) .. controls (280,153.7) and (284.7,149) .. (290.5,149) .. controls (296.3,149) and (301,153.7) .. (301,159.5) .. controls (301,165.3) and (296.3,170) .. (290.5,170) .. controls (284.7,170) and (280,165.3) .. (280,159.5) -- cycle ;

\draw   (460,159.5) .. controls (460,153.7) and (464.7,149) .. (470.5,149) .. controls (476.3,149) and (481,153.7) .. (481,159.5) .. controls (481,165.3) and (476.3,170) .. (470.5,170) .. controls (464.7,170) and (460,165.3) .. (460,159.5) -- cycle ;

%Shape: Square [id:dp05663682075637111] 
\draw   (197.5,148.5) -- (218,148.5) -- (218,168.5) -- (197.5,168.5) -- cycle ;

%Shape: Square [id:dp05663682075637111] 
\draw   (377.5,128.5) -- (398,128.5) -- (398,148.5) -- (377.5,148.5) -- cycle ;

%Shape: Square [id:dp05663682075637111] 
\draw   (377.5,168.5) -- (398,168.5) -- (398,188.5) -- (377.5,188.5) -- cycle ;

%Shape: Square [id:dp05663682075637111] 
\draw   (557.5,108.5) -- (578,108.5) -- (578,128.5) -- (557.5,128.5) -- cycle ;

%Shape: Square [id:dp05663682075637111] 
\draw   (557.5,148.5) -- (578,148.5) -- (578,168.5) -- (557.5,168.5) -- cycle ;

%Shape: Square [id:dp05663682075637111] 
\draw   (557.5,188.5) -- (578,188.5) -- (578,208.5) -- (557.5,208.5) -- cycle ;

%Straight Lines [id:da7070219457632385] 
\draw    (121,159.5) -- (197.5,159.5) ;

%Straight Lines [id:da7070219457632385] 
\draw    (301,161.5) -- (377.5,179.5) ;

%Straight Lines [id:da7070219457632385] 
\draw    (301,157.5) -- (377.5,139.5) ;

%Straight Lines [id:da7070219457632385] 
\draw    (481,156) -- (557.5,118.5) ;

%Straight Lines [id:da7070219457632385] 
\draw    (481,159.5) -- (557.5,159.5) ;

%Straight Lines [id:da7070219457632385] 
\draw    (481,163) -- (557.5,198.5) ;
%Straight Lines [id:da891242303205514] 

% Text Node
\draw (70,152) node [anchor=north west][inner sep=0.75pt]   [align=left] {$T_3:$};
% Text Node
\draw (250,152) node [anchor=north west][inner sep=0.75pt]   [align=left] {$T_4:$};
% Text Node
\draw (430,152) node [anchor=north west][inner sep=0.75pt]   [align=left] {$T_5:$};
% Text Node
\draw (105.5,152) node [anchor=north west][inner sep=0.75pt]   [align=left] {1};
% Text Node
\draw (286,152) node [anchor=north west][inner sep=0.75pt]   [align=left] {1};
% Text Node
\draw (465,152) node [anchor=north west][inner sep=0.75pt]   [align=left] {1};
% % Text Node
\draw (203.5,151) node [anchor=north west][inner sep=0.75pt]   [align=left] {3};
% % Text Node
\draw (563,152) node [anchor=north west][inner sep=0.75pt]   [align=left] {1};
% % Text Node
\draw (563,112) node [anchor=north west][inner sep=0.75pt]   [align=left] {1};
% % Text Node
\draw (563,192) node [anchor=north west][inner sep=0.75pt]   [align=left] {1};
% % Text Node
\draw (383,171) node [anchor=north west][inner sep=0.75pt]   [align=left] {1};
% % Text Node
\draw (383,131) node [anchor=north west][inner sep=0.75pt]   [align=left] {2};

% % Text Node
\draw (155,145) node [anchor=north west][inner sep=0.75pt]   [align=left] {3};
% % Text Node
\draw (335,133) node [anchor=north west][inner sep=0.75pt]   [align=left] {2};
% % Text Node
\draw (514,120) node [anchor=north west][inner sep=0.75pt]   [align=left] {1};
% % Text Node
\draw (335,173) node [anchor=north west][inner sep=0.75pt]   [align=left] {1};
% % Text Node
\draw (514,145) node [anchor=north west][inner sep=0.75pt]   [align=left] {1};

% % Text Node
\draw (514,184) node [anchor=north west][inner sep=0.75pt]   [align=left] {1};
% % Text Node
% \draw (147,219) node [anchor=north west][inner sep=0.75pt]   [align=left] {6};
\end{tikzpicture}
\end{center}
We have
\begin{align*}
\Cont_{T_1}&=-\left(y^9+y^{-9}-\frac{1}{2}\right);\\
\Cont_{T_2}&=y^9+y^{-9}-2;\\
\Cont_{T_3}&=\left(y^{9/2}-y^{-9/2}\right)^2\left(-\frac{1}{(y^{3/2}-y^{-3/2})^2}-\frac{10y+10y^{-1}+7}{(y^{1/2}-y^{-1/2})^2}\right);\\
\Cont_{T_4}&=9(y^3-y^{-3})^2(y^{3/2}-y^{-3/2})^2\left(\frac{1}{2(y-y^{-1})^2}-\frac{2}{(y^{1/2}-y^{-1/2})^2}\right)\frac{1}{(y^{1/2}-y^{-1/2})^2};\\
\Cont_{T_5}&=-\frac{9(y^{3/2}-y^{-3/2})^6}{2(y^{1/2}-y^{-1/2})^6}.\\
\end{align*}
Combining with the fact that 
\[n_{0,6}=-17064,\, n_{1,6}=80948,\, n_{2,6}=-194022,\, n_{3,6}=290853,\,n_{4,6}=-290400,\]
\[n_{5,6}=196857,\, n_{6,6}=-90390,\, n_{7,6}=27538,\, n_{8,6}=-5310,\, n_{9,6}=585,\,n_{10,6}=-28,\]
we may deduce that $\Omega_{6}^{\mathbb{P}^2}$ equals
\begin{align*}
\frac{1}{y^{10}}(1&+y+4y^2+7y^3+16y^4+25y^5+47y^6+68y^7+104y^8+128y^9+146y^{10}+128y^{11}\\
&+104y^{12}+68y^{13}+47y^{14}+25y^{15}+16y^{16}+7y^{17}+4y^{18}+y^{19}+y^{20})\frac{y^9-y^{-9}}{y^{1/2}-y^{-1/2}}\,.
\end{align*}

This method allows us to determine $\Omega_{d}^{\mathbb{P}^2}$ once the Gopakumar-Vafa invariants $n_{g,d'}$ for $d'\leq d$ are known. In the following table, we provide a list of $\hat{\Omega}^{\mathbb P^2}_d$
for $d\leq 10$, where $\hat{\Omega}^{\mathbb P^2}_d$ is given by \eqref{eqn:shiftOmega}. 
Actually, we have computed  $\hat{\Omega}^{\mathbb P^2}_d$ for $d\leq 16$ using  Gopakumar-Vafa invariants of local $\PP^2$
in \cite{HKR}. However, as the coefficients of $\hat{\Omega}^{\mathbb P^2}_d$ grow rapidly, we limit the list to $d\leq 10$ in this paper. A complete list of $\hat{\Omega}^{\mathbb P^2}_d$ for $d\leq 16$ is available at:

\vspace{1em}
\centerline{\url{https://sites.google.com/site/guoshuaimath/poincarepolynomials}}
\vspace{1em}

\noindent
From our calculation, we make the following conjecture: 
\begin{conj}
For $3| d$,
$$
\frac{\hat{\Omega}^{\mathbb P^2}_d}{y^2+y+1}
$$
is a palindromic polynomial of $y$.
\end{conj}

\newpage

\begin{table}[htb]  
\begin{center}   
\caption{List of $\hat{\Omega}^{\mathbb P^2}_d$}
\begin{tabular}{|c|c|}  
\hline $d$ & $\hat{\Omega}^{\mathbb P^2}_d$\\
\hline $1$ & \footnotesize $1$\\
\hline $2$ & \footnotesize $1$\\
\hline $3$ & \footnotesize $y^2+y+1$ \\
\hline $4$ & \footnotesize $y^6+y^5+4y^4+4y^3+4y^2+y+1$\\
\hline $5$ & \footnotesize 
$y^{12}+y^{11}+4y^{10}+7y^9+13y^8+19y^7+23y^6+19y^5+13y^4+7y^3+4y^2+y+1$\\
\hline $6$ &  \tabincell{c}{ \footnotesize 
$(y^2+y+1)(y^{18}+3 y^{16}+4 y^{15}+9 y^{14}+12 y^{13}+26 y^{12}+30 y^{11}+48 y^{10}$ \\ \footnotesize  $+ 50 y^9+48 y^8+30 y^7+26 y^6+12 y^5+9 y^4+4 y^3+3 y^2+1$) }\\
\hline $7$ & \tabincell{c}{ \footnotesize 
$y^{30}+y^{29}+4 y^{28}+7 y^{27}+16 y^{26}+28 y^{25}+53 y^{24}+86 y^{23}+146 y^{22}+225 y^{21}+342 y^{20}$ \\ \footnotesize  $+ 489 y^{19}+674 y^{18}+859 y^{17}+1018 y^{16}+1076 y^{15}+1018 y^{14}+859 y^{13}+674 y^{12}+489 y^{11}$ \\
\footnotesize  $+342 y^{10}+ 225 y^9+146 y^8+86 y^7+53 y^6+28 y^5+16 y^4+7 y^3+4 y^2+y+1$ }\\
\hline $8$ & \tabincell{c}{ \footnotesize 
$y^{42}+y^{41}+4 y^{40}+7 y^{39}+16 y^{38}+28 y^{37}+56 y^{36}+92 y^{35}+164 y^{34}+261 y^{33}+429 y^{32}+ 654 y^{31}$ \\ \footnotesize  $+1007 y^{30}+1463 y^{29}+2129 y^{28}+2947 y^{27}+4024 y^{26}+5236 y^{25}+6616 y^{24}+7856 y^{23}+8846 y^{22}$ \\
\footnotesize  $+9166 y^{21}+8846 y^{20}+7856 y^{19}+6616 y^{18}+5236 y^{17}+4024 y^{16}+2947 y^{15}+2129 y^{14}+1463 y^{13}$ \\
\footnotesize $+1007 y^{12}+ 654 y^{11}+429 y^{10}+261 y^9+164 y^8+92 y^7+56 y^6+28 y^5+16 y^4+7 y^3+4 y^2+y+1
$}\\
\hline $9$ & \tabincell{c}{ \footnotesize 
$(y^{2}+y+1)(y^{54}+3 y^{52}+4 y^{51}+9 y^{50}+15 y^{49}+32 y^{48}+48 y^{47}+90 y^{46}+141 y^{45}+234 y^{44} $ \\ \footnotesize  $+360 y^{43}+572 y^{42}+843 y^{41}+1275 y^{40}+1840 y^{39}+2652 y^{38}+3711 y^{37}+5159 y^{36}+6954 y^{35}$ \\
\footnotesize  $+9291 y^{34}+12073 y^{33}+15354 y^{32}+18948 y^{31}+22685 y^{30}+25965 y^{29}+28395 y^{28}+29271 y^{27}$ \\
\footnotesize  $+28395 y^{26}+25965 y^{25}+22685 y^{24}+18948 y^{23}+15354 y^{22}+12073 y^{21}+9291 y^{20}+6954 y^{19}$ \\
\footnotesize  $+5159 y^{18}+3711 y^{17}+2652 y^{16}+1840 y^{15}+1275 y^{14}+843 q^{13}+572 q^{12}+360 q^{11}$ \\
\footnotesize $+234 y^{10}+141 y^9+90 y^8+48 y^7+32 y^6+15 y^5+9 y^4+4 y^3+3 y^2+1)
$}\\
\hline $10$ & \tabincell{c}{ \footnotesize 
$y^{72}+y^{71}+4 y^{70}+7 y^{69}+16 y^{68}+28 y^{67}+56 y^{66}+95 y^{65}+173 y^{64}+285 y^{63}+483 y^{62}+771 y^{61} $ \\ \footnotesize  $+1247 y^{60}+1928 y^{59}+2996 y^{58}+4507 y^{57}+6763 y^{56}+9901 y^{55}+14423 y^{54}+20579 y^{53}+29168 y^{52}$ \\
\footnotesize  $+40605 y^{51}+56058 y^{50}+76158 y^{49}+102495 y^{48}+135818 y^{47}+178022 y^{46}+229643 y^{45}+292339 y^{44}$ \\
\footnotesize  $+365554 y^{43}+449335 y^{42}+540160 y^{41}+634381 y^{40}+723486 y^{39}+799099 y^{38}+849619 y^{37}$ \\
\footnotesize  $+867997 y^{36}+849619 y^{35}+799099 y^{34}+723486 y^{33}+634381 y^{32}+540160 y^{31}+449335 y^{30}$ \\
\footnotesize  $+365554 y^{29}+292339 y^{28}+229643 y^{27}+178022 y^{26}+135818 y^{25}+102495 y^{24}+76158 y^{23}$ \\
\footnotesize  $+56058 y^{22}+40605 y^{21}+29168 y^{20}+20579 y^{19}+14423 y^{18}+9901 y^{17}+6763 y^{16}+4507 y^{15}$ \\
\footnotesize  $+2996 y^{14}+1928 y^{13}+1247 y^{12}+771 y^{11}+483 y^{10}+285 y^9+173 y^8$ \\
\footnotesize $+95 y^7+56 y^6+28 y^5+16 y^4+7 y^3+4 y^2+y+1
$}\\
\hline
\end{tabular}
\end{center}   
\end{table}

\newpage

\appendix
\section{Higher range asymptotic formulas for (refined) Poincaré polynomials (by Miguel Moreira)}\label{sec:higherrange}
Throughout this appendix, we denote by $M_{d,\chi}$ the moduli space of 1-dimensional stable sheaves on $\BP^2$ with support of degree $d$ and Euler characteristic $\chi$, with the assumption that $\gcd(d,\chi)=1$. Our goals are:

\begin{itemize}
\item Explaining the relation between the formulas for the leading Betti numbers found in the present paper and the approach to calculate the cohomology rings of moduli spaces $M_{d, \chi}$ in \cite{PS23, klmp} (Corollary \ref{cor: relations}).
\item Proposing a conjecture strengthening Conjecture \ref{conj:2ndappro} by allowing a larger range (Conjecture \ref{conj: higherrange}).
\item Proposing a conjecture involving the refinements from the perverse/Chern filtration (Conjecture \ref{conj: higherrangerefined}).
\end{itemize}

\subsection{Generators and relations description of $H^\ast(M_{d,\chi})$} \label{subsection:geometry}
The natural generators of the cohomology are the tautological classes, obtained by taking Kunneth components of Chern classes of the universal bundle. More precisely, we define
\[c_k(j)=p_\ast\big(\ch_{k+1}(\BF)q^\ast H^j\big)\in H^{2(k+j-1)}(M_{d, \chi})\,\]
for $k\geq 0$ and $j=0,1,2$. Here, $p, q$ are the projections of $M_{d,\chi}\times \mathbb{P}^2$ onto $M_{d,\chi}$ and $\mathbb{P}^2$, respectively, and $\BF$ is a (rational, normalized) universal sheaf on the product $M_{d,\chi}\times \mathbb{P}^2$. The universal sheaf $\BF$ is normalized in a way that $c_1(1)=0$; we refer to \cite[Section 1.1]{PS23} for more details on the normalization. A classical argument by Beauville shows that tautological classes generate the cohomology:

\begin{prop}[\cite{PS23, B95}]
The classes $c_k(j)$ generate $H^\ast(M_{d,\chi})$ as a $\BQ$-algebra.
\end{prop}

In other words, the natural algebra homomorphism from the free polynomial algebra
\begin{equation}
\label{eq: realization}
\BD\coloneqq \BQ[c_2(0), c_0(2), c_3(0), c_2(1), c_1(2), \ldots] \to H^\ast(M_{d, \chi})
\end{equation}
is surjective. The algebra $\BD$ is sometimes referred to as the descendent algebra. The descendent algebra has a natural cohomological grading making the homomorphism above respect gradings. Note that $\BD$ is infinite dimensional, but the subspaces $\BD^{2j}$ of degree $2j$ elements are finite dimensional, and we have\footnote{This is almost the same infinite product that appears in Theorem \ref{thm:leadingBetti}, except that the exponent in $1/(1-y)$ is 2 instead of 1. The difference is explained by the fact that we are not dividing by $\frac{1-y^{3d}}{1-y}$.}
\[H(y)\coloneqq \sum_{j\geq 0} y^j \dim \BD^{2j}= \prod_{k>0} \frac{1}{(1-y^k)^2(1-y^{k+1})}\,.\]
In \cite{PS23, klmp} the authors construct families of geometric relations among tautological classes in $H^\ast(M_{d,\chi})$. The most important relations are the so called generalized Mumford relations. They come from the following simple observation \cite[Proposition 2.6]{klmp}: if $F, F'$ are semistable sheaves with topological types $(d, \chi)$ and $(d', \chi')$ satisfying 
\begin{equation}\label{eq: ineqslopes}
\frac{\chi'}{d'}<\frac{\chi}{d}<\frac{\chi'}{d'}+3
\end{equation}
then $\Hom(F, F')=\Ext^2(F, F')=0$. Hence, $\Ext^1(F, F')$ has constant dimension $dd'$, and therefore there is a vector bundle $\CV$ of rank $dd'$ on the product $M_{d,\chi}\times M_{d',\chi'}$ whose fiber over $(F, F')$ is $\Ext^1(F, F')$. The generalized Mumford relations are obtained from the vanishing of  Chern classes beyond the rank of $\CV$:
\[c_k(\CV)=0 \textup{ in }H^\ast(M_{d,\chi}\times M_{d',\chi'})\textup{ for }k>dd'\,.\]
We get relations
\[\GMR_{d, \chi}^{d', \chi', k, A}\in \ker\big(\BD\to H^\ast(M_{d,\chi})\big)\]
for each $d, \chi, d', \chi'$ satisfying \eqref{eq: ineqslopes}, $k>dd'$ and $A\in H^\ast(M_{d',\chi'})$ by integrating $c_k(\CV)$ along the Poincaré dual of $A$. The generalized Mumford relation above has cohomological degree
\[2\big(k+\deg(A)-((d')^2+1)\big)\geq 2d'(d-d')\]
with equality when $A$ is the unit class and $k=dd'+1$. 

There are 2 other families of geometric relations introduced in \cite{klmp}, but they only appear in cohomological degree quadratic in $d$. These families of geometric relations do not necessarily generate the ideal of relations -- for example, in $M_{5,1}$ there are three relations in degrees 36, 38 which are not in the ideal of geometric relations.

\subsubsection{The $H^{< 2(d-1)}$ range} The first generalized Mumford relations that appear are when $d'=1$ or $d'=d-1$. The cohomological degree of $\GMR$ relations in either case starts at $2(d-1)$. Indeed, there are no relations among generators below such degree since, by the result of Yuan \cite{Y23},
\[\dim H^{2j}(M_{d,\chi})=[y^j]H(y)=\dim \BD^{2j}\]
for $j\leq d-2$.

\subsubsection{The $H^{< 2(2d-4)}$ range} In the range between $2(d-1)$ and $2(2d-4)$ we have the effect of generalized Mumford relations, but only with\footnote{A priori, $\GMR$ with $d'=d-1$ could also play a role. However, it turns out that such relations are already contained in the ideal of $\GMR$ with $d'=1$ in the range considered.} $d'=1$. When we take $\chi=1$ and $d'=1$, the inequality \eqref{eq: ineqslopes} is equivalent to $\chi'\in \{0, -1, -2\}$. Moreover, $M_{1, \chi'}$ is isomorphic to $\BP^2$ (or, more canonically, it is isomorphic to the dual of the original $\BP^2$), so $A$ is taken from $\{h^i\colon i\in \{0,1,2\}\}$ where $h\in H^2(M_{1, \chi'})$ is the generator. 

If we assemble into a generating series the degrees of all the possible generalized Mumford relations with fixed $d, \chi, d'=1$ and $k, \chi', A$ varying through the possibilities discussed, we get
\[\sum_{k>d}\sum_{\chi'=-2}^0\sum_{i=0}^2 y^{k-2+i}=3y^{d-1}\frac{(1+y+y^2)}{1-y}\,.\]

This is exactly the rational function appearing in Theorem \ref{thm:leadingBetti}. The following proposition has been obtained in joint work of M.M. with  Lim and Pi:

\begin{prop}[Lim--Pi--Moreira]
The generalized Mumford relations with $d'=1$ are completely independent in $\BD^{\leq 2(2d-4)}$. More precisely, if $j\leq 2d-4$ then the elements
\[D\cdot \GMR_{d, \chi}^{d'=1, \chi', k, h^i}\in \BD^{2j}\]
are linearly independent, where $\chi', k, i$ run through the possibilities discussed above and $D$ runs through monomials in tautological classes with $\deg(D)+2(k-2+i)=2j$.
\end{prop}

A consequence of the asymptotic formula of Theorem \ref{thm:leadingBetti} is then the following:

\begin{cor}\label{cor: relations}
For $j< 2d-10$ (or, assuming the better bound in Remark \ref{rmk:betterbd}, $j<2d-4$) the ideal of relations 
\[\ker\big(\BD^{2j}\to H^{2j}(M_{d, \chi})\big)\]
is (freely) generated by the generalized Mumford relations with $d'=1$. 
\end{cor}

\subsubsection{Higher range}
When extending from $2(2d-4)$ to $2(3d-9)$, as in the proposed formula in Conjecture \ref{conj:2ndappro}, two things happen. First, $\GMR$ relations with $d'=2$ play a role; secondly, there are redundancies in the relations. While this makes it hard to give such a simple explanation for the formula as in the $<2(2d-4)$ range, let us rewrite Conjecture~\ref{conj:2ndappro} in a form that is more natural.

This rewriting is inspired by the recursion for the Betti numbers of the moduli spaces of stable bundles on curves in \cite{HN, AB}, which can be explained via the Harder--Narasimhan stratification of the stack of all (not necessarily stable) vector bundles. Indeed, this stratification is closely related to generalized Mumford relations, and provides a way to study their redundancies, see for example \cite[Proposition 3.3]{lmpcurves}.

Let $\widetilde P_{d,\chi}(y)$ be the Poincaré series of the stack $\FM_{d, \chi}$ of semistable sheaves, i.e.
\[\widetilde P_{d, \chi}(y)\coloneqq \sum_{j\geq 0}y^j \dim H^{2j}(\FM_{d,\chi})\in \BZ[\![y]\!]\,.\]
When $\gcd(d, \chi)=1$, $\FM_{d, \chi}\simeq M_{d, \chi}\times B\BC^\times$ and 
\[\widetilde P_{d, \chi}(y)=\frac{1}{1-y}P_d(y)\,.\]
In general, $\widetilde P_{d, \chi}$ depends only on $d$ and $\gcd(d,\chi)$, and it can be recovered from knowing $P_{d/k}(y)$ for every $k$ dividing $\gcd(d, \chi)$ by a plethystic exponential type formula, see \cite[Corollary 4.9]{klmp}. For example,
\begin{align*}\widetilde P_{2,0}(y)&=\frac{P_2(y)}{1-y}+\frac{1}{2}y\left(\frac{P_1(y)}{1-y}\right)^2-\frac{1}{2}y\frac{P_1(y^2)}{1-y^2}\,\\
&=\frac{(1 + y + y^2)(1 + y^2 + y^3 + y^4 - y^5)}{(1-y)(1-y^2)}\,.
\end{align*}
The rewriting of Conjecture \ref{conj:2ndappro} depends on the following observation:
\[3f(y)=-3\widetilde P_{2,1}-3\widetilde P_{2,0}+6y\widetilde P_{1,0}^2\,.\]
In this formula, it is natural to interpret the term $3\widetilde P_{2,1}$ as coming from $\GMR$ with $d'=2$ and $\chi'=-1, -3, -5$ and $3\widetilde P_{2,0}$ coming from\footnote{Generalized Mumford relations can also be defined when $\gcd(d', \chi')\neq 1$, as long as one works with the stack of semistables.} $\GMR$ with $d'=2$ and $\chi'=0, -2, -4$. The term $6y\widetilde P_{1,0}^2$ might be interpreted as being related to redundancies among the relations. Then, Conjecture \ref{conj:2ndappro} is equivalent to the recursion
\begin{align}\label{eq: recursionrange3}P_d&+3y^{d-1}\widetilde P_{1,0}P_{d-1}+3 y^{2d-4} \widetilde P_{2,1}P_{d-2}\\
&+3 y^{2d-4}\widetilde P_{2,0}P_{d-2}+3 y^{2d-3}\widetilde P_{1,0}^2P_{d-2}=H(y) \mod y^{3d-9}\,. \nonumber
\end{align}
This way of rewriting Conjecture \ref{conj:2ndappro} actually leads us to propose a more general conjecture. To state it, let $\mathsf{HN}_k$ be the set of ``Harder--Narasimhan types'', consisting of pairs $(\boldsymbol d, \boldsymbol \chi)$ where $\boldsymbol d=(d_1, \ldots, d_m)\in \BZ_{>0}^m$ and $\boldsymbol \chi=(\chi_1, \ldots, \chi_m)\in \BZ^m$ satisfy
\[d_1+ \ldots+ d_m\leq k\textup{ and }0\leq \frac{\chi_1}{d_1}<\ldots <\frac{\chi_m}{d_m}<3\,.\]
Set $d_0=d-\sum_{i=1}^m d_i$ and let
\[s(\boldsymbol d)\coloneqq \sum_{0\leq i < j \leq m}d_i d_j\,.\]

\begin{conj}\label{conj: higherrange}
Let $k\geq 0$. For $d> k+1$ we have
\[\sum_{(\bf{d}, \boldsymbol{\chi})\in\mathsf{HN}_k}y^{s(\boldsymbol d)}P_{d_0}\widetilde P_{d_1, \chi_1}\ldots \widetilde P_{d_m, \chi_m}=H(y) \mod y^{(k+1)(d-k-1)}\]
where the sum runs over the set of ``Harder--Narasimhan types''. 
\end{conj}

The conjecture implies the existence of rational functions $f_0, f_1, \ldots$ with the property that
\[P_d(y)=H(y)\left(\sum_{j=0}^k y^{j(d-j)}f_j(y)\right) \mod y^{(k+1)(d-k-1)}\,.\]
for every $d>k+1$. The rational functions $f_k$ only have poles at roots of unity. Indeed, it follows from \cite[Theorem 4.11]{klmp} that the denominator of $f_k$ is $(1-y)\ldots (1-y^k)$. The rational function $f_k$ can explicitly determined from $P_1, \ldots, P_k$. The first ones $f_0, f_1, f_2$ are already implicit in Conjecture \ref{conj:2ndappro}. The next one is
\begin{align*}f_3(y)=\frac{-3}{(1-y)(1-y^2)(1-y^3)}&\big(9 + 18 y - 44 y^3 - 82 y^4 - 37 y^5 + 56 y^6 + 143 y^7 
+ 170 y^8 \\&+ 
   164 y^9 + 125 y^{10} + 89 y^{11} + 55 y^{12} + 36 y^{13} + 18 y^{14} + 
   9 y^{15}\big)\,.
\end{align*}

Conjecture \ref{conj: higherrange} suggests that the generalized Mumford relations are complete at least up to cohomological degree $2\lfloor d^2/4\rfloor-1$.
\subsection{Perverse/Chern refinement}

One of the reasons why the study of the cohomology of $M_{d, \chi}$ is interesting is that it gives direct access to Gopakumar--Vafa invariants of $K_{\BP^2}$, according to a proposal by Maulik--Toda \cite{MT18}. They propose the definition of such invariants in terms of the perverse filtration
\[P_0H^\ast(M_{d, \chi})\subseteq  P_1H^\ast(M_{d, \chi})\subseteq P_2H^\ast(M_{d, \chi})\subseteq \ldots\]
on $M_{d,\chi}$ associated to the Hilbert--Chow morphism
\[M_{d,\chi}\to |dH|\]
sending a 1-dimensional sheaf to its fitting support. The perverse filtration determines the refined Poincaré polynomial
\[P^{\mathsf{ref}}_d(q,t)\coloneqq \sum_{i, j\geq 0}q^i t^j \dim \gr_j^P H^{i+j}(M_{d,\chi})\,,\]
and Maulik--Toda propose the definition of $n_{g,d}^{\mathsf{MT}}$ in terms of the $t=1$ specialization:
\[P^{\mathsf{ref}}_d(q,1)=\sum_{g\geq 0}n_{g,d}^{\mathsf{MT}}(q^{1/2}-q^{-1/2})^{2g}\,.\]
The Gromov--Witten/Gopakumar--Vafa correspondence is the conjecture that these agree with $n_{g,d}$ defined via the Gromov-Witten theory of $K_{\BP^2}$, or equivalently $P^{\mathsf{ref}}_d(q,1)=\mathcal F_d(q)$. From this point of view, the all-genus local/relative correspondence in the present paper can be phrased as a surprising relation between the specializations $q=t=y^{1/2}$ (which recovers $P_d(y)$) and $q=y, t=1$ of $P^{\mathsf{ref}}_d$.

It has been conjectured in \cite{KPS23, klmp} that the perverse filtration matches another natural filtration called the Chern filtration $C_\bullet H^\ast(M_{d,\chi})$, which is defined in terms of tautological classes:
\[C_k H^\ast(M_{d,\chi})\coloneqq \textup{span}\{c_{k_1}(j_1)\ldots c_{k_m}(j_m)\colon k_1+\ldots+k_m\leq k\}\,.\]
Note that the perverse filtration can be defined in $\BD$ exactly in the same way, and we let
\[H^{\mathsf{ref}}(q,t)\coloneqq \sum_{i, j\geq 0}q^i t^j \dim \gr_j^P \BD^{i+j}=\prod_{k>0}\frac{1}{(1-q^{k-1}t^{k+1})(1-q^{k+1}t^{k-1})(1-q^{k+1}t^{k+1})}\,.\]
The $P=C$ conjecture has been shown \cite{msy, PSSZ24} to hold in the free range of the cohomology, i.e.
\[P_\bullet H^{\leq 2d-4}(M_{d,\chi})=C_\bullet H^{\leq 2d-4}(M_{d,\chi})\,.\]
This equality implies the asymptotic formula
\[P^{\mathsf{ref}}_d(q,t)=H^{\mathsf{ref}}(q,t) \mod (q,t)^{2d-2}\,.\]
From the point of view of the Chern filtration and the asymptotic formulas in the present paper, it becomes very natural to ask if there are refinements of the higher asymptotic formulas discussed here for the $q=t=y^{1/2}$ specialization. Numerical evidence\footnote{M.M. thanks Y. Kononov and W. Pi for sharing the (conjectural) data for the refined Poincaré polynomials, calculated via Nekrasov partition functions as in \cite[Section 3.4]{KPS23}.} suggests that this is indeed the case. To state our refined conjecture, one defines the refined Poincaré polynomials $\widetilde{P}^{\mathsf{ref}}_{d,\chi}$ of the stacks, which can still be calculated from $P^{\mathsf{ref}}_{d}$ with a plethystic exponential type formula \cite[Remark 6.11]{klmp}.
For example, 
\begin{align*}\widetilde P^{\mathsf{ref}}_{2,0}(q,t)&=\frac{P^{\mathsf{ref}}_2(q,t)}{1-qt}+\frac{1}{2}qt\bigg(\frac{P^{\mathsf{ref}}_1(q,t)}{1-qt}\bigg)^2-\frac{1}{2}qt\frac{P^{\mathsf{ref}}_1(q^2, t^2)}{1-q^2t^2}\,\\
&=\frac{(1+t^2+t^4)(1 + q t^3 + t^6 + q^2 t^6 - q^2 t^8)}{(1-qt)(1-q^2t^2)}\,.
\end{align*}
 We refer the reader to \cite{davison} for details on the definition of the perverse filtration on the stacks $\FM_{d,\chi}$; the plethystic formula is a consequence of \cite[(56)]{davison}.
 
Given $(\boldsymbol d, \boldsymbol \chi)\in \mathsf{HN}_k$, we let
\[s_\pm(\boldsymbol d, \boldsymbol \chi)\coloneq s(\boldsymbol d)\pm \sum_{i=1}^m \chi_i\,.\]

 \begin{conj}\label{conj: higherrangerefined}
Let $k\geq 0$. For $d> k+1$ we have
\[\sum_{(\boldsymbol{d}, \boldsymbol{\chi})\in\mathsf{HN}_k}q^{s_+(\boldsymbol d, \boldsymbol \chi)}t^{s_-(\boldsymbol d, \boldsymbol \chi)}P_{d_0}^{\mathsf{ref}}\widetilde P_{d_1, \chi_1}^{\mathsf{ref}}\ldots \widetilde P_{d_m, \chi_m}^{\mathsf{ref}}=H^{\mathsf{ref}}(q,t) \mod (q,t)^{2(k+1)(d-k-1)}\]
where the sum runs over the set of ``Harder-Narasimhan types''. 
\end{conj}
 
%When $k=3$ the conjecture reads as
%\begin{align*}P^{\mathsf{ref}}_d&+q^{d-1}t^{d-3}(q^2+qt+t^2)\widetilde{P}^{\mathsf{ref}}_{1,0}P^{\mathsf{ref}}_{d-1}+q^{2d-3}t^{2d-9}(q^4+q^2t^2+t^4) \widetilde{P}^{\mathsf{ref}}_{2,1}P^{\mathsf{ref}}_{d-2}\\
%&+q^{2d-4}t^{2d-8}(q^4+q^2t^2+t^4)\widetilde{P}^{\mathsf{ref}}_{2,0}P^{\mathsf{ref}}_{d-2}+q^{2d-2}t^{2d-6}(q^2+qt+t^2)(\widetilde{P}^{\mathsf{ref}}_{1,0})^2 P^{\mathsf{ref}}_{d-2}\\&\quad =H^{\mathsf{ref}}(q,t) \mod (q,t)^{6d-18}\,.
%\end{align*}

As in the unrefined case, the conjecture is equivalent to asymptotic formulas of the form
\[P_d^{\mathsf{ref}}(q,t)=H^{\mathsf{ref}}(q,t)\left(\sum_{j=0}^k q^{j(d-j)}t^{j(d-j-3)+1}f_j^{\mathsf{ref}}(q,t)\right) \mod (q,t)^{2(k+1)(d-k-1)}\,\]
for every $d>k+1$, where $f_0^{\mathsf{ref}}=1, f_1^{\mathsf{ref}},f_2^{\mathsf{ref}}, \ldots$ are rational functions in $q,t$. The rational function $f_k^{\mathsf{ref}}$ can be explicitly calculated from $P^{\mathsf{ref}}_1, \ldots, P^{\mathsf{ref}}_k$. For example,
\begin{align*}
f_1^{\mathsf{ref}}(q,t)&=-\frac{(q^2 + q t + t^2) (1 + t^2 + t^4)}{1-qt}\\
f_2^{\mathsf{ref}}(q,t)&=-\frac{(q^2 + q t + t^2) (1 + t^2 + t^4)}{1-qt}\big(q^3 + t^3- q^3 t^2   - q t^4- q^5 t^2 - q^3 t^4 - 2 q^2 t^5\\
&\quad  - 
 q^4 t^5 
 - q^2 t^7 - q t^8 - q^3 t^8 - q^5 t^8 + t^9 - q^2 t^{11}\big)\,.
\end{align*}

\begin{rmk}
Theorem \ref{thm:main0} does not have a refined version. For example, the polynomials $P_d^{\mathsf{ref}}(q,t)$ are irreducible for $d=3,6,9,12$. On the other hand, they are divisible by $t^4+t^2+1$ for $d\leq 12$ not a multiple 3 (and the quotient is irreducible for $d\neq 1,2$). It is natural to speculate that this divisibility holds in general.
\end{rmk}

\bibliographystyle{amsxport}
\bibliography{universal-BIB}

\end{document}